\documentclass[hidelinks,onefignum,onetabnum]{siamart251216}

\usepackage{lipsum}
\usepackage{amsfonts}
\usepackage{graphicx}
\usepackage{epstopdf}
\usepackage{algorithmic}
\ifpdf
  \DeclareGraphicsExtensions{.eps,.pdf,.png,.jpg}
\else
  \DeclareGraphicsExtensions{.eps}
\fi

\newsiamremark{remark}{Remark}
\newsiamremark{hypothesis}{Hypothesis}
\crefname{hypothesis}{Hypothesis}{Hypotheses}
\newsiamthm{claim}{Claim}

\headers{Robust-Pathwise Control of Elliptic Equations Under Uncertainty}{Q. Sun, Y. Cui, L. Guo, and L. Ju}

\title{Learning Optimal Robust-Pathwise Control of Elliptic Equations with Random Coefficients\thanks{\funding{Qi Sun was partially supported by the Natural Science Foundation of China under grant number 12671528 and 12201465. Ling Guo was partially supported by the Natural Science Foundation of China under grant number 92270115 and 12071301. }}}

\author{%
  Qi Sun$^{\dagger,\ddagger}$\!\!\!
  \and 
  Yanhan Cui\thanks{School of Mathematical Sciences, Tongji University, Shanghai 200092, China (\email{qsun\_irl@tongji.edu.cn}, \email{2433968@tongji.edu.cn}).\protect\\[0.8pt]
  \hspace*{1.33em}$^\ddagger$Key Laboratory of Intelligent Computing and Applications (Ministry of Education), Tongji University, Shanghai 200092, China.\protect\\[0.8pt]
  \hspace*{1.33em}$^\S$Department of Mathematics, Shanghai Normal University, Shanghai 200234, China (\email{lguo@shnu.edu.cn}).\protect\\[0.8pt]
  \hspace*{1.33em}$^\P$Department of Mathematics, University of South Carolina, Columbia, SC, 29208, USA (\email{ju@math.sc.edu}).}
  \and
  Ling Guo$^\S$\!\!\!
  \and 
  Lili Ju$^\P$
}

\usepackage{amsopn}

\usepackage{color}
\usepackage{tikz} 
\usetikzlibrary{arrows,shapes,positioning,automata}
\usetikzlibrary{calc,decorations.markings}
\usetikzlibrary{decorations.pathreplacing}
\usepackage{adjustbox}
\usepackage{mathtools} 
\usepackage{rotating}
\usepackage{multirow,multicol}
\usepackage{colortbl}
\definecolor{kugray5}{RGB}{224,224,224}
\usepackage{makecell}
\newcommand{\ms}{\mkern 2mu} 

\begin{document}

\maketitle

\begin{abstract}
Conventional control approaches to optimization problems under uncertainty entail a practical compromise: deterministic controls offer robustness at the expense of pathwise accuracy while stochastic controls can adapt to individual realizations but suffer from the curse of dimensionality. To achieve both robustness and adaptivity, a neural operator-based approach is proposed in this paper for distributed elliptic optimal control problems with high-dimensional random inputs. A data-driven training algorithm is first established to learn the mapping from coefficient realizations directly to their optimal pathwise controls. To bypass the burden of generating training data pairs, a physics-informed training algorithm is then proposed and recast as weakly supervised learning using pseudo label targets produced by standard adjoint-based control updates. Numerical experiments with rough coefficients are conducted to validate the effectiveness and efficiency of our robust-pathwise controls, which outperforms the benchmark robust deterministic control on unseen coefficient realizations owing to the generalization capability and fast inference of trained models.
\end{abstract}



\section{Introduction}

The spatial heterogeneity inherent in material properties, as exemplified by additive manufacturing and reservoir exploitation \cite{smith2024uncertainty,gunzburger2014stochastic}, serves as a principal source of aleatory uncertainty in engineering applications. Often a prompt and robust decision needs to be made across a wide range of realizations, motivating the design of deterministic control strategies for partial differential equations (PDEs) with uncertain coefficients \cite{heinkenschloss2025optimization,lee2017comparison,borzi2010treatment,van2019robust,sun2020distributed,babuska2004galerkin,frauenfelder2005finite,charrier2012strong}. To further enhance adaptability to individual realizations, stochastic control approaches have been formulated to depend explicitly on random parameters \cite{chen2014weighted,benner2016block,ali2017multilevel,dai2025solving,yin2024aonn}, offering the potential for superior performance compared to deterministic counterparts. However, the parameter space induced by the representation of random inputs is typically high-dimensional \cite{ghanem2003stochastic,graham2011quasi}, rendering traditional polynomial-based numerical methods computationally intractable due to the curse of dimensionality. While contemporary advancements in scientific machine learning \cite{karniadakis2021physics,raissi2019physics,li2020fourier,lu2021learning,luo2025efficient} have emerged as promising alternatives \cite{dai2025solving,yin2024aonn,hwang2022solving}, significant impediments persist, particularly concerning excessive training overheads and limited model interpretability.

More specifically, we consider a distributed optimal control problem governed by elliptic PDEs with uncertain coefficients \cite{charrier2012strong} and stochastic control functions \cite{chen2014weighted,yin2024aonn,dai2025solving,benner2016block}. Given a deterministic target function $U(x):D\to\mathbb{R}$, the control objective is formulated as the minimization of a tracking-type cost functional
\begin{equation}
	\mathcal{J}(f) = \mathbb{E}\left[\frac12\int_D|u(x,\omega)-U(x)|^2dx+\frac{\beta}{2}\int_D|f(x,\omega)|^2dx\right]
	\label{Sto-Ctl-Objective-Strong}
\end{equation}
subject to a second-order elliptic boundary value problem with uncertain coefficients $a(x,\omega)$, namely, for almost every (a.e.) $\omega$ in a set of outcomes $\Omega$,
\begingroup
\renewcommand*{\arraystretch}{1.4}
\begin{equation}
\Bigg\{
\begin{array}{cl}
-\nabla\cdot\left(a(x,\omega)\nabla u(x,\omega)\right)=f(x,\omega),\ \ & \textnormal{in}\ D,\\
u(x,\omega)=0,\ \ & \textnormal{on}\ \partial D.
\end{array}
\label{Sto-Ctl-Constraint-Strong}
\end{equation}
\endgroup
Here, and in what follows, $D\subset\mathbb{R}^d$ represents the spatial domain, $\partial D$ its boundary, $\beta>0$ a regularization parameter,  $``\nabla"$ means differentiation with respect to (w.r.t.) the spatial variable $x\in D$ only, and there are positive $a_{\textnormal{min}}$ and $a_{\textnormal{max}}$ such that
\begin{equation}
	a_{\textnormal{min}} \le a(x,\omega) \le a_{\textnormal{max}}, \ \ \textnormal{a.e.}\ (x,\omega) \in D \otimes\Omega,
	\label{Assumption-Coeff}
\end{equation}
Note that the state function $u(x,\omega)\in L^2(\Omega;H_0^1(D))$ in \eqref{Sto-Ctl-Objective-Strong} is uniquely determined by the stochastic control function $f(x,\omega)\in L^2(\Omega; L^2(D))$ via equation \eqref{Sto-Ctl-Constraint-Strong}, hence the objective functional is expressed in terms of the control function for brevity.

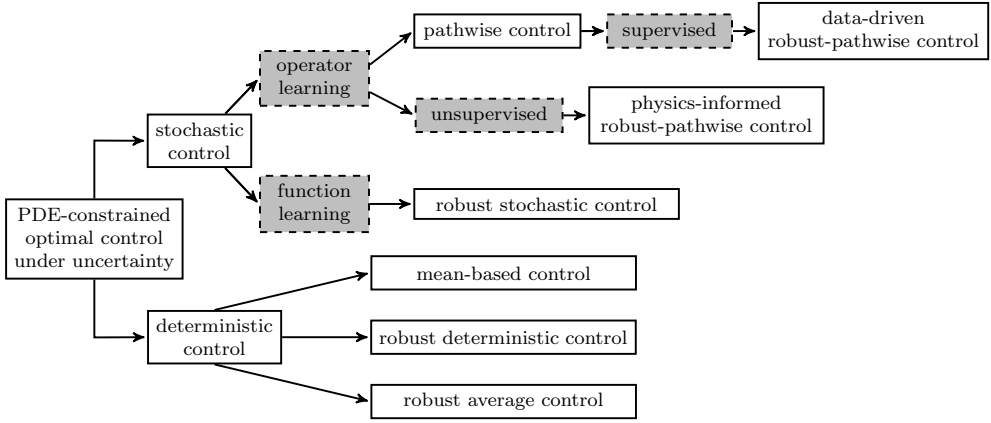
\begin{figure}[t!]
\begin{adjustbox}{max totalsize={\textwidth}{\textheight},center}
\begin{tikzpicture}[shorten >=1pt,auto,node distance=2.5cm,thick,main node/.style={rectangle,draw,font=\footnotesize},decoration={brace,mirror,amplitude=7},dimension node/.style={draw,dashed,fill=lightgray,font=\footnotesize},decoration={brace,mirror,amplitude=7}]

\node[main node] (1) {\parbox{2.33cm}{\centering PDE-constrained \\ optimal control \\ under uncertainty}};
\node[draw=none,fill=none] (2) [right = 0.1cm of 1] {};
\node[main node] (3) [above = .9cm of 2] {\parbox{1.27cm}{\centering stochastic \\ control}};
\node[main node, xshift=0.22cm] (4) [below = .9cm of 2] {\parbox{1.7cm}{\centering deterministic \\ control}};
\node[dimension node] (5) [above right = 0.14cm of 3] {\parbox{1.35cm}{\centering operator \\ learning}};
\node[dimension node] (15) [below right = 0.14cm of 3] {\parbox{1.35cm}{\centering function \\ learning}};
\node[main node] (16) [right = 0.63cm of 15] {\parbox[c][0.2cm]{3.61cm}{\centering  robust stochastic control}};
\node[main node, xshift=1cm] (11) [above right = 0.4cm of 4] {\parbox[c][0.25cm]{3.61cm}{\centering mean-based control}};
\node[main node] (12) [right = 1.28cm of 4] {\parbox[c][0.25cm]{3.61cm}{\centering robust deterministic control }};
\node[main node, xshift=1cm] (14) [below right = 0.4cm of 4] {\parbox[c][0.25cm]{3.61cm}{\centering  robust average control}};

\path[draw,->,>=stealth'] (1.north) -- ++(0,.84cm) -> (3.west);
\path[draw,->,>=stealth'] (1.south) -- ++(0,-.84cm) -> (4.west);
\path[every node/.style={font=\sffamily\small,sloped},->,>=stealth']
    (3) edge [above] (5.west)
    (3) edge [below] (5.west)
    (4.north) edge (11.west)
    (4) edge (12.west)
    (4.south) edge (14.west);
\path[every node/.style={font=\sffamily\small,sloped},->,>=stealth']
	(3) edge [above] (15.west)
	(3) edge [below] (15.west)
	(15) edge (16);

\node[draw=none,fill=none] (17) [right = 1.5cm of 5] {};
\node[dimension node, xshift=0.1cm] (8) [below = 0.15cm of 17] {\parbox{1.9cm}{\centering unsupervised}};
\node[main node] (7) [above = 0.3cm of 17, xshift = 0.22cm] {\parbox{2.18cm}{\centering pathwise control}};
\node[main node] (16) [right = 0.36cm of 8] {\parbox{3.17cm}{\centering physics-informed \\ robust-pathwise control}};
\node[dimension node] (9) [right = 0.36cm of 7] {\parbox{1.58cm}{\centering supervised}};
\node[main node] (6) [right = 0.36cm of 9] {\parbox{3.1cm}{\centering data-driven \\ robust-pathwise control}};

\path[every node/.style={font=\sffamily\small,sloped},->,>=stealth']
	([yshift=0.2cm]5.east) edge (7.west)
	([yshift=-0.2cm]5.east) edge (8.west)
    (7)	edge (9)
    (7)	edge (9)
    (8) edge (16)
    (9) edge (6);    	
\end{tikzpicture}
\end{adjustbox}
\vspace{-0.6cm}
\caption{Control strategies for PDE-constrained optimization under high-dimensional uncertainty. }
\label{fig-big-picture}
\vspace{-0.3cm}
\end{figure}

The primary bottleneck in numerically solving this minimization problem is the curse of dimensionality, as the random coefficient often exhibits limited spatial regularity that results in a high truncation number within its Karhunen--Lo$\grave{\text{e}}$ve expansion \cite{ghanem2003stochastic}. To bypass the approximation of a high-dimensional stochastic control function, various deterministic control strategies have been proposed in the literature \cite{lee2017comparison,sun2020distributed}, as outlined in \autoref{fig-big-picture} and discussed further in section~\ref{Sec-Ctrl-Overview}. Owing to its robustness against variations in the coefficient field, the robust deterministic control \cite{borzi2010treatment,borzi2010multigrid,lee2017comparison} is widely employed in practice and frequently serves as a golden benchmark. However, deterministic control strategies cannot adapt to individual realizations, which would compromise the overall tracking performance. Alternatively, the pathwise control \cite{ali2017multilevel} provides satisfactory accuracy by decoupling the stochastic optimal control problem into independent trajectories, but necessitates an iterative online solving process for every unseen coefficient realization.

To combine the strengths of both robustness and adaptivity, this work proposes a neural operator-based control strategy that learns the solution mapping from individual coefficient realizations directly to their optimal pathwise control functions. A key observation is that the trained network parameters are robust against variations in the random coefficient field, whereas the output control prediction adapts to each input coefficient realization even under high-dimensional stochastic parametrization. Therefore, this paradigm is termed robust-pathwise control in contrast to established control approaches in the literature \cite{lee2017comparison,ali2017multilevel,hou2011finite,gunzburger2014stochastic}, as illustrated by the schematic in \autoref{fig-big-picture}. Rather than focusing on the design of neural operators \cite{lu2021learning,li2020fourier,jia2025novel,azizzadenesheli2024neural}, this work explores the training algorithm of the control surrogate to facilitate a seamless transition from supervised (data-driven) to unsupervised (physics-informed) learning regimes. The effectiveness and efficiency of our proposed algorithms are numerically validated against the widely used robust deterministic control \cite{lee2017comparison,sun2020distributed,heinkenschloss2025optimization}, achieving superior state tracking performance alongside real-time inference.

The remainder of this paper is organized as follows. Section~\ref{Sec-Preliminaries} provides the representation of random coefficients utilizing the Karhunen--Lo$\grave{\text{e}}$ve expansion, followed by a comparative review of both deterministic and stochastic control strategies. The solution mapping/operator for the robust stochastic control problem is then introduced, which motivates the employment of neural surrogate model. Section~\ref{Sec-RpCtrl-data-driven} and Section~\ref{Sec-RpCtrl-physics-informed} illustrate the supervised and unsupervised training of our proposed robust-pathwise control model, followed by computational experiments in Section~\ref{Sec-Num-Exps} that validate the efficiency and effectiveness of our approaches. Finally, some concluding remarks and future directions are reported in Section 6.


\section{Preliminaries and related work} \label{Sec-Preliminaries}

\subsection{Representation of uncertain coefficients} 

It is commonly assumed that the available information regarding $g(x,\omega)=\log( a(x,\omega) )$ is limited to its mean
\begin{equation*}
	E_g(x) = \int_\Omega g(x,\omega)\,dP(\omega)= 0
\end{equation*} 
and covariance
\begin{equation}
	\textnormal{Cov}_g(x,y) = \int_\Omega \big( g(x,\omega)-E_g(x) \big) \big( g(y,\omega)-E_g(y) \big) dP(\omega) = \sigma^2 e^{-\frac{\lVert x-y\rVert}{\ell}},
	\label{Sto-Coeff-Cov}
\end{equation}
in which the variance $\sigma^2>0$ indicates the magnitude of fluctuations, the correlation length $\ell>0$ indicates the frequency of oscillation, and $\lVert\cdot\rVert$ denotes a distance metric defined for any two points in the spatial domain $\bar{D}$.

To facilitate numerical implementation, the Karhunen--Lo$\grave{\text{e}}$ve expansion \cite{ghanem2003stochastic,frauenfelder2005finite} is widely used to separate spatial dependencies from stochastic variables, as detailed in the following proposition.

\begin{proposition}\textnormal{\cite{ghanem2003stochastic}}
The covariance operator of a stochastic function $g(x,\omega)\in L^2(\Omega;L^2(D))$, namely,
\begin{equation*}
	\mathcal{C}_g: L^2(D) \to L^2(D), \qquad (\mathcal{C}_g\varphi)(x) = \int_D \textnormal{Cov}_g(x,y)\varphi(y)\,dy,
\end{equation*}
is a symmetric, non-negative and compact integral operator which admits a countable sequence $\{(\lambda_m,\varphi_m)\}_{m\geq 1}$ of eigenpairs (arranged in decreasing order) satisfying
\begin{equation}
	\int_{D\times D} \textnormal{Cov}_g(x,y)\varphi_m(y)v(x)\,dydx = \lambda_m\int_D\varphi_m(x)v(x)\,dx \ \ \ \textnormal{for any}\ v\in L^2(D).
	\label{KLE-Eigenprob}
\end{equation}

Moreover, there exists a sequence of random variables $\{\theta_m(\omega)\}_{m\geq 1}$ such that 
\begin{equation*}
	\mathbb{E}[\theta_m(\omega)]=0 \ \ \ \textnormal{and} \ \ \ \mathbb{E}[\theta_m(\omega)\theta_n(\omega)]=\delta_{mn} \ \ \  \textnormal{for any}\ m,n\geq 1,
\end{equation*}
and $g(x,\omega)$ can be expanded in a Karhunen--Lo$\grave{\text{e}}$ve expansion:
\begin{equation}
	g(x,\omega) = \mathbb{E}[g] + \sum_{m\geq 1}\sqrt{\lambda_m}\varphi_m(x)\theta_m(\omega).
	\label{KLE-Log-Coeff}
\end{equation}
Note that, when the series is truncated after a finite number, $M\in\mathbb{Z}_+$, of terms, the resulting approximation $g_M(x,\omega)$ minimizes the error
\begin{equation*}
	\lVert g(x,\omega) - g_M(x,\omega) \rVert_{L^2(\Omega;L^2(D))}^2 = \sum_{m>M}\lambda_m.
\end{equation*}
\end{proposition}

The truncation number is usually determined through a cumulative energy criterion, which selects the smallest integer $M$ such that the energy ratio
\begin{equation}\label{KLE-Eigenprob-Ratio}
	\rho_M={\sum_{m=1}^M \lambda_m}/{\sum_{m=1}^\infty \lambda_m}	
\end{equation}
exceeds a user-defined tolerance value. Notably, for covariance functions \eqref{Sto-Coeff-Cov} with low regularity \cite{frauenfelder2005finite}, the slow decay of eigenvalues necessitates a large number of retained modes to meet the energy criterion, leading to a high-dimensional parameter space susceptible to the curse of dimensionality.

\subsection{Overview of established control strategies}\label{Sec-Ctrl-Overview}

Various formulations for elliptic PDE-constrained optimization problems with uncertain coefficients have been proposed in the literature \cite{lee2017comparison,ali2017multilevel,heinkenschloss2025optimization,luo2025efficient}, which can be broadly categorized according to the use of deterministic or stochastic control functions as depicted in \autoref{fig-big-picture}.

In practice, deterministic control functions are widely adopted in engineering applications due to their robustness against random inputs and ease of implementation, with differences arising in their cost functionals or constraint equations. More specifically, three representative cases are compared and presented as follows:
\begin{itemize}
\item[(a)] Mean-based control \cite{borzi2010multigrid,borzi2010treatment} reduces to a deterministic optimization problem by substituting uncertain coefficients with the expected value $\bar{a}(x)=\mathbb{E}[ a(x,\omega) ]$ 
\vspace{-.25em}
\begin{equation*}
\begin{array}{cl}
\mathop{\mathrm{minimize}}\limits_{\bar{f}(x)} & \displaystyle \bar{\mathcal{J}} ( \bar{f} ) = \frac12 \|\bar{u}(x) - U(x) \|^2_{L^2(D)} + \frac{\beta}{2} \|\bar{f}(x) \|^2_{L^2(D)} \\[0.8em]
\textnormal{subject to} & 
\left\{
\begin{array}{cl}
\displaystyle - \nabla \cdot \left( \bar{a}(x)\nabla \bar{u}(x)\right) = \bar{f}(x),\ \ & \textnormal{in}\ D,\\
\displaystyle \bar{u}(x)=0,\ \ & \textnormal{on}\ \partial D.
\end{array}\right.
\end{array}
\end{equation*}
However, such a formulation overlooks uncertainties in the coefficients, yielding a control function that often fails to respond to input variations \cite{lee2017comparison}.
\item[(b)] Robust deterministic control \cite{borzi2010treatment,hou2011finite,lee2017comparison,van2019robust,sun2020distributed,heinkenschloss2025optimization,gunzburger2011optimal} couples different coefficient realizations through a unified cost functional 
\vspace{-.25em}
\begin{equation*}
\begin{array}{cl}
\mathop{\mathrm{minimize}}\limits_{\hat{f}(x)} & \displaystyle \hat{\mathcal{J}} ( \hat{f} ) = \mathbb{E} \left[ \frac12 \| \hat{u}(x,\omega)-U(x)\|^2_{L^2(D)} \right] + \frac{\beta}{2} \|\hat{f}(x)\|^2_{L^2(D)} \\[0.8em]
\textnormal{subject to} & 
\left\{
\begin{array}{cl}
\displaystyle - \nabla \cdot \left( a(x,\omega)\nabla \hat{u}(x,\omega)\right) = \hat{f}(x),\ \ & \textnormal{in}\ D \otimes\Omega,\\
\displaystyle \hat{u}(x,\omega)=0,\ \ & \textnormal{on}\ \partial D \otimes \Omega.
\end{array}\right.
\end{array}
\end{equation*}
The integration of numerous coefficient realizations enables the control function to achieve robustness against random fluctuations, thereby making it an applicable benchmark in practice \cite{lee2017comparison}.
\item[(c)] Robust average control \cite{lee2017comparison,gunzburger2008error} prioritizes the expectation tracking objective over penalizing large deviations arising from particular realizations, with the loss term related to that of (b) through the Jensen's inequality
\vspace{-.25em}
\begin{equation*}
\begin{array}{cl}
\mathop{\mathrm{minimize}}\limits_{\check{f}(x)} & \displaystyle \check{\mathcal{J}} ( \check{f} ) = \frac12 \left\| \mathbb{E}[\check{u}(x,\omega)]-U(x) \right\|^2_{L^2(D)} + \frac{\beta}{2} \|\check{f}(x)\|^2_{L^2(D)} \\[0.8em]
\textnormal{subject to} & 
\left\{
\begin{array}{cl}
\displaystyle - \nabla \cdot \left( a(x,\omega)\nabla \check{u}(x,\omega)\right) = \check{f}(x),\ \ & \textnormal{in}\ D \otimes\Omega,\\
\displaystyle \check{u}(x,\omega)=0,\ \ & \textnormal{on}\ \partial D \otimes \Omega.
\end{array}\right.
\end{array}
\end{equation*}
\end{itemize}
Owing to its practical feasibility and robustness against input uncertainties \cite{lee2017comparison,heinkenschloss2025optimization}, the robust deterministic control approach (b) is extensively applied to address PDE-constrained optimization problems under uncertainty, making it an appropriate baseline for comparison with our proposed control surrogates.

To overcome the lack of adaptability in deterministic control strategies \cite{lee2017comparison}, the stochastic control function is formulated to depend explicitly on random parameters \cite{chen2014weighted,ali2017multilevel}, among which two notable approaches are as follows:
\begin{itemize}
\item[(d)] Robust stochastic control \cite{borzi2010multigrid,borzi2010treatment,hou2011finite} facilitates adjustments to individual realizations, thereby offering the potential to achieve superior performance
\vspace{-.25em}
\begin{equation*}
\begin{array}{cl}
\mathop{\mathrm{minimize}}\limits_{f(x,\omega)} & \displaystyle \mathcal{J} ( f ) = \mathbb{E} \left[ \frac12 \| u(x,\omega) - U(x) \|^2_{L^2(D)} + \frac{\beta}{2} \| f(x,\omega) \|^2_{L^2(D)} \right] \\[0.8em]
\textnormal{subject to} & 
\left\{
\begin{array}{cl}
\displaystyle - \nabla \cdot \left( a(x,\omega)\nabla u(x,\omega)\right) = f(x,\omega),\ \ & \textnormal{in}\ D \otimes\Omega,\\
\displaystyle u(x,\omega)=0,\ \ & \textnormal{on}\ \partial D \otimes \Omega.
\end{array}\right.
\end{array}
\end{equation*}
Stochastic Galerkin or collocation methods are standard for numerical simulation but suffer from the curse of dimensionality \cite{gunzburger2014stochastic}, motivating the development of neural network-based function learning approaches \cite{dai2025solving,yin2024aonn,song2024admm}.
\item[(e)] Pathwise control \cite{ali2017multilevel} applies the Monte Carlo method to mitigate the curse of dimensionality, decomposing (d) into independent realizations (hereafter, the subscript $\omega$ is used to indicate a specific realization of a stochastic field)
\vspace{-.25em}
\begin{equation*}
\begin{array}{cl}
\mathop{\mathrm{minimize}}\limits_{f_\omega(x)} & \displaystyle \mathcal{J}_\omega ( f_\omega ) = \frac12 \| u_\omega(x) - U(x) \|^2_{L^2(D)} + \frac{\beta}{2} \| f_\omega(x) \|^2_{L^2(D)} \\[0.8em]
\textnormal{subject to} & 
\left\{
\begin{array}{cl}
\displaystyle - \nabla \cdot \left( a_\omega(x)\nabla u_\omega(x)\right) = f_\omega(x),\ \ & \textnormal{in}\ D,\\
\displaystyle u_\omega(x)=0,\ \ & \textnormal{on}\ \partial D.
\end{array}\right.
\end{array}
\end{equation*}
However, solutions tailored to specific realizations fail to generalize to unseen instances, while applying their mean value reduces to deterministic control. 
\end{itemize}

\subsection{Feedback control via observable inputs} \label{Sec-Notations}

As discussed in the previous section, the robust stochastic control reveals itself to be a compelling strategy owing to its pathwise adaptivity to individual coefficient realizations \cite{hou2011finite,lee2017comparison,gunzburger2008error}. However, the random parameter is typically a latent variable that eludes direct measurement in real-world applications \cite{christie2001tenth}, rendering the stochastic control function practically infeasible for the governing system \eqref{Sto-Ctl-Constraint-Strong}. This limitation underscores the need for a feedback control driven by observable inputs, the mechanism of which is elucidated by the introduction of solution mapping presented below.

To be specific, let $H^r(D)$ ($r\in\mathbb{Z}_+$) denote the standard Sobolev space equipped with the inner product $(\cdot,\cdot)_{H^r(D)}$ and the induced norm $\| \cdot\|_{H^r(D)}$ \cite{evans2022partial,brenner2007mathematical,charrier2012strong}, and let $H^r_0(D)=\{v\in H^r(D): v=0 \text{ on } \partial D\}$. For a $\sigma$-finite probability space $(\Omega,\mathcal{F},\mu)$ \cite{frauenfelder2005finite} with a $\sigma$-algebra of events $\mathcal{F}$ and a probability measure $\mu : \mathcal{F}\rightarrow[0,1]$, we define the stochastic Sobolev space $L^2(\Omega;H_0^r(D))$ \cite{babuska2004galerkin,frauenfelder2005finite,charrier2012strong} with the inner product
\begin{equation*}
	\left[u,v\right]_{L^2(\Omega;H_0^r(D))}=\mathbb{E} \big[ (u,v)_{H_0^r(D)} \big]=\int_\Omega(u,v)_{H_0^r(D)}d\mu(\omega)
\end{equation*}
and the induced norm 
\begin{equation*}
	\lVert u\rVert ^2_{L^2(\Omega;H_0^r(D))}=\mathbb{E}\big[\lVert u\rVert^2_{H_0^r(D)}\big]=\int_\Omega\lVert u\rVert^2_{H_0^r(D)}d\mu(\omega).
\end{equation*}
Note that the Hilbert space $L^2(\Omega;H_0^r(D))$ is isomorphic to the tensor product space $H_0^r(D)\otimes L^2(\Omega)$ \cite{babuska2004galerkin}, and the function space $L^2(\Omega;L^2(D))$ is constructed analogously.

Then, by defining two bilinear forms 
\begin{equation*}
	b[u,v]=\mathbb{E}\left[\int_D a \nabla u\cdot\nabla v\, dx\right]\ \ \textnormal{and}\ \ [f,v]=\mathbb{E}\left[\int_Df v\, dx\right]
\end{equation*}
for any $u,v\in L^2(\Omega;H_0^1(D))$ and $f\in L^2(\Omega;L^2(D))$, the weak formulation of \eqref{Sto-Ctl-Constraint-Strong} is 
\begin{equation}
	b[u,v]=[f,v]\ \ \ \textnormal{for any}\ v\in L^2(\Omega;H_0^1(D)),
	\label{Sto-Ctl-Constraint-Weak}
\end{equation}
where the existence and uniqueness of its weak solution are established below.

\begin{lemma} \label{Lemma-weakBVP-Existence}
Under assumption \eqref{Assumption-Coeff}, the elliptic equation \eqref{Sto-Ctl-Constraint-Weak} admits a unique weak solution $u(x,\omega)\in L^2(\Omega;H_0^1(D))$ for any given $f(x,\omega) \in L^2(\Omega;L^2(D))$ \textnormal{\cite{charrier2012strong,hou2011finite}}. 
\end{lemma}

Accordingly, the robust stochastic control problem can be reformulated as
\begin{equation}
	\textnormal{seek}\ f\in L^2(\Omega;L^2(D))\ \textnormal{such that}\ \eqref{Sto-Ctl-Objective-Strong}\ \textnormal{is minimized subject to}\ \eqref{Sto-Ctl-Constraint-Weak},
	\label{Sto-Ctl-Weak}
\end{equation}
from which the optimality system is derived via the Lagrange multiplier rule \cite{hou2011finite}.

\begin{lemma} \label{Lemma-Sto-Ctrl-Existence}
Under assumption \eqref{Assumption-Coeff}, the distributed optimization problem \eqref{Sto-Ctl-Weak} admits a unique optimal control $f^*(x,\omega) \in L^2(\Omega;L^2(D))$ satisfying 
\begin{equation} \label{Sto-Ctl-Optimality}
    \begin{array}{c l}
        b[u^*,v] = [f^*,v], & \ \text{for any } v \in L^2(\Omega;H_0^1(D)), \\[1ex]
        b[\xi^*,\eta] = [u^*-U,\eta], & \ \text{for any } \eta \in L^2(\Omega;H_0^1(D)), \\[1ex]
        [\beta f^* + \xi^*, t-f^*] = 0, & \ \text{for any } t \in L^2(\Omega;L^2(D)),
    \end{array}
\end{equation}
where $u^*(x,\omega) \in L^2(\Omega;H_0^1(D))$ and $\xi^*(x,\omega) \in L^2(\Omega;H_0^1(D))$ denote the corresponding state and adjoint state, respectively \textnormal{\cite{hou2011finite}}.
\end{lemma}

In other words, Lemma \ref{Lemma-Sto-Ctrl-Existence} presents a nonlinear solution operator that maps the random coefficient field to its optimal stochastic control
\begin{equation}\label{Solution-Mapping-StoCtrl}
	\mathcal{S}: L^\infty(\Omega;L^\infty(D)) \to L^2(\Omega;L^2(D)), \qquad  a(x,\omega) \mapsto f^*(x,\omega).
\end{equation}
Unfortunately, numerically resolving this solution mapping through traditional grid-based methods is computationally prohibitive due to the curse of dimensionality \cite{gunzburger2008error,ghanem2003stochastic}, motivating the employment of neural operators as computational surrogates \cite{azizzadenesheli2024neural}. However, existing literature \cite{wang2021fast,hwang2022solving,luo2025efficient} typically relies on a two-phase paradigm: an offline phase that trains a neural operator \cite{li2020fourier,lu2021learning} to map the coefficient and control fields to their associated system state, namely, $\{a(x,\omega),f(x,\omega)\} \mapsto u(x,\omega)$, followed by an online phase that solves the optimal control problem through a gradient-based algorithm. While accelerating the repeated state and adjoint evaluations \cite{lundqvist2025employing}, it still requires an iterative search to identify the optimal control. In contrast, our methods directly learn the solution mapping \eqref{Solution-Mapping-StoCtrl}, therefore enabling instantaneous inference for any coefficient realization. Accordingly, developing effective and efficient training algorithms has emerged as a critical challenge \cite{song2024admm}, which is the focus of this work.

It is also noteworthy that several recent works \cite{yin2024aonn,dai2025solving,song2024admm,yuan2025adaptive} directly compute the Karush-Kuhn-Tucker system \eqref{Sto-Ctl-Optimality} via a function learning approach \cite{raissi2019physics}. Though effective for approximating high-dimensional functions, the resulting optimal control function depends explicitly on the latent random parameters, which is often impractical to observe in real-world engineering applications \cite{christie2001tenth}. Besides, neural operators have recently been applied to optimization problems governed by dynamical systems \cite{feng2025optimal,lundqvist2025employing}, emerging as an effective and efficient alternative in conjunction with neural network-based approximation methods \cite{li2024neural}.

\section{Data-driven robust-pathwise control}\label{Sec-RpCtrl-data-driven}

This section introduces the data-driven training of our proposed robust-pathwise control (see \autoref{figure-flowchart-RpCtrl-data-driven}). We begin by outlining the generation of a high-quality training dataset utilizing the Monte Carlo finite element method, and subsequently present the detailed training algorithm.

\begin{figure}[t!]
\centering
\includegraphics[width=0.8\textwidth]{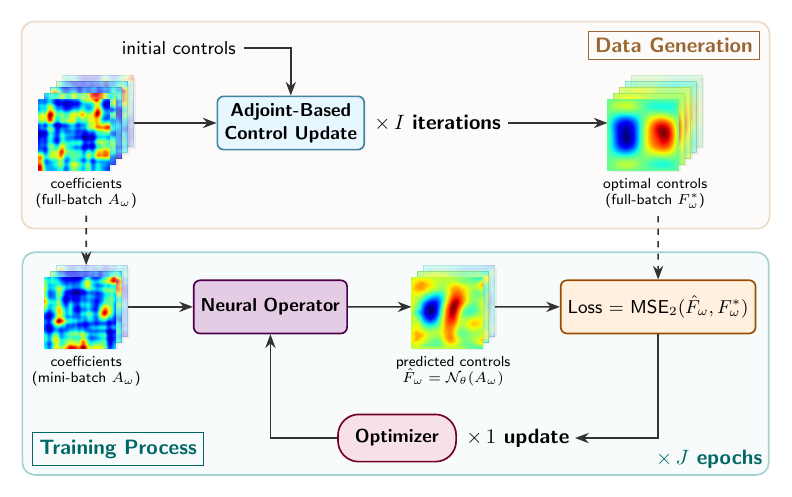}
\vspace{-0.5cm}
\caption{Supervised learning flowchart for our proposed data-driven robust-pathwise control.}
\label{figure-flowchart-RpCtrl-data-driven}
\vspace{-0.3cm}
\end{figure}

\subsection{Monte Carlo finite element discretization} \label{Section-MC-FEM}

To construct a neural operator surrogate for the solution mapping \eqref{Solution-Mapping-StoCtrl}, the Monte Carlo method is adopted to mitigate the curse of dimensionality, leading to the pathwise solution operator 
\begin{equation}\label{Solution-Mapping-StoCtrl-Pathwise}
	\mathcal{S}_\omega: L^\infty(D) \to L^2(D), \qquad a_\omega(x) \mapsto f_\omega^*(x),
\end{equation}
with realizations drawn independently and identically from the probability space. It is also noteworthy that the existence of the pathwise solution operator follows from the standard pathwise optimal control problem \cite{ali2017multilevel,lions1971optimal}.

Next, the finite element method \cite{brenner2007mathematical,hou2011finite} is employed for spatial discretization, as the supervised training of neural operators typically requires finite-dimensional data pairs \cite{winovich2019convpde,li2020fourier,lu2021learning}. Accordingly, the numerical counterpart defines a discrete mapping between finite-dimensional spaces
\begin{equation}\label{Solution-Mapping-StoCtrl-Pathwise-Numerical}
	\mathcal{S}_{h,\omega}: V_h \to U_h, \qquad a_{h,\omega}(x) \mapsto f_{h,\omega}^*(x),
\end{equation}
where $V_h \subset L^\infty(D)$ and $U_h \subset L^2(D)$ indicate the finite element spaces on a uniform triangulation $\mathcal{T}_h$ of domain $D$ with mesh width $h>0$. Here, both the coefficient and control functions are discretized using identical conforming basis functions
\begin{equation*}
	V_h = U_h = \textnormal{span}\{\phi_1(x),\cdots,\phi_N(x)\},
\end{equation*}
which naturally bypasses any input-output dimensionality mismatch during the construction of neural operators.

For the numerical representation of coefficient realizations, the Galerkin approximation of \eqref{KLE-Eigenprob} is introduced: find $\lambda_m^h\in\mathbb{R}_+$ and $\varphi^h_m(x) \in V_h$ such that
\begin{equation}
	\int_{D\times D} \textnormal{Cov}_g(x,y)\varphi^h_m(y)v_h(x)\,dydx = \lambda_m^h\int_D\varphi_m^h(x)v_h(x)\,dx, \ \ \ \textnormal{for any}\ v_h\in V_h,
	\label{KLE-Generalized-Eigenprob}
\end{equation}
where $\varphi^h_m(x) = \sum_{n=1}^N c_{mn}\phi_n(x)$ with unknowns $c_{mn}\in\mathbb{R}$ to be computed by solving the algebraic generalized eigenvalue problem \cite{babuska2004galerkin,ghanem2003stochastic}. Upon truncating the expansion at $M\in\mathbb{Z}_+$ terms to ensure the retained energy ratio exceeds a prescribed threshold, the finite-dimensional approximation of the coefficient realization is given by 
\begin{equation}
\begin{aligned}
	a_\omega(x) & \approx \textnormal{exp}\bigg(\sum_{m=1}^{M}\sqrt{\lambda_m^h}\varphi_m^h(x)\theta_m(\omega)\bigg) &\!\!\! \textnormal{(truncation)} \\
	& \approx \sum_{n=1}^N \textnormal{exp}\bigg(\sum_{m=1}^{M}\sqrt{\lambda_m^h}\varphi_m^h(x_n)\theta_m(\omega)\bigg) \phi_n(x) &\!\!\! \textnormal{(interpolation)} \\
	& = \sum_{n=1}^N \textnormal{exp}\bigg( \sum_{m=1}^{M} \sqrt{\lambda_m^h} \bigg(\sum_{n'=1}^N c_{m{n'}} \phi_{n'}(x_n) \bigg) \theta_m(\omega)  \bigg) \phi_n(x) &\!\!\! \textnormal{(nodal basis)} \\
	& = \sum_{n=1}^N \textnormal{exp} \bigg( \sum_{m=1}^{M} c_{mn} \sqrt{\lambda_m^h} \theta_m(\omega) \bigg) \phi_n(x) = a_{h,\omega}(x) &
\end{aligned}	
	\label{Sto-Coeff-KLE}
\end{equation}
which relies on the Kronecker delta property $\phi_{n'}(x_n)=\delta_{n'n}$ for $1\leq n',n\leq N$.

To numerically obtain the optimal pathwise control for the coefficient realization \eqref{Sto-Coeff-KLE}, the classical gradient descent method \cite{sun2020distributed,troltzsch2010optimal} is employed and terminated at a prescribed number of iterations, yielding the finite-dimensional representation
\begin{equation}\label{Pathwise-OptCtrl-Discrete}
	f_{h,\omega}^*(x) = \sum_{n=1}^N f_{n,\omega}^*\phi_n(x)
\end{equation}
for the solution mapping \eqref{Solution-Mapping-StoCtrl-Pathwise-Numerical}. For brevity, the detailed iterative process is deferred to the data generation stage in Algorithm~\ref{Alg-RpCtrl-DataDriven} or \autoref{figure-flowchart-RpCtrl-data-driven} \cite{sun2020distributed,troltzsch2010optimal}, with notations
\begin{equation*}
	b_k(u_h,v_h)= \int_D a_{h,\omega_k} \nabla u_h \cdot\nabla v_h\, dx\ \ \textnormal{and}\ \ (f_h,v_h) = \int_D f_h v_h\, dx
\end{equation*}
defined for any deterministic functions $u_h$, $v_h$, and $f_h\in V_h$.

\subsection{Supervised neural operator learning} 

Once the optimal pathwise control \eqref{Pathwise-OptCtrl-Discrete} is computed for its given coefficient realization \eqref{Sto-Coeff-KLE}, the discrete mapping \eqref{Solution-Mapping-StoCtrl-Pathwise-Numerical} can be recast as an image-to-image regression task
\begin{equation}\label{Solution-Mapping-StoCtrl-Pathwise-Img2Img}
	\mathcal{S}_{h,\omega}: \mathbb{R}^{B \times 1 \times H\times W} \to \mathbb{R}^{B \times 1 \times H\times W}, \qquad A_\omega \mapsto F_\omega^*,
\end{equation}
where the nodal values for both $a_{h,\omega}(x)$ and $f_{h,\omega}^*(x)$ are geometrically reshaped into the image-like tensors $A_\omega$ and $F_\omega^*\in \mathbb{R}^{B \times 1\times H\times W}$ according to the uniform triangulation $\mathcal{T}_h$ of spatial domain $D$ \cite{jia2025novel}. Here, the input and output tensors take the shape $B \times 1\times H\times W$ as shown in \autoref{figure-flowchart-RpCtrl-data-driven}, where $B$, $1$, $H$, and $W$ denote the batch size, number of channels, spatial height, and spatial width, respectively \cite{jia2025novel,hwang2022solving}. Besides, to ensure a stable training procedure, both input and output tensors are normalized prior to network optimization (not relabelled for notational simplicity).

With the robust stochastic control problem \eqref{Sto-Ctl-Weak} converted to a regression task, we proceed to parametrize the discrete mapping \eqref{Solution-Mapping-StoCtrl-Pathwise-Img2Img} using a neural operator
\begin{equation}\label{RpCtrl-data-driven}
	\mathcal{N}_\theta(A_\omega) = F_\omega^*,
\end{equation}
where $\theta\in\mathbb{R}^p$ denotes the collection of trainable parameters and the neural operator can be constructed using various modern architectures, including the U-Net \cite{winovich2019convpde,jia2025novel}, deep operator networks \cite{lu2021learning}, Fourier neural operator \cite{li2020fourier}, and their variants.

\begin{algorithm}[t!]
\caption{Data-Driven Robust-Pathwise Control}
\begin{algorithmic}
\STATE \textbf{Input:} step sizes $\eta,\tau>0$, maximum iteration count $I$ and epoch count $J$.
\STATE \textbf{Output:} data-driven robust-pathwise control $\mathcal{N}_{\theta^*}$.

\STATE \textit{\% Data Generation}
\STATE -- compute $\{ \lambda_m^h, \varphi^h_m(x) \}_{m\geq 1}$ via the generalized eigenvalue problem \eqref{KLE-Generalized-Eigenprob};
\STATE -- generate $\{a_{h,\omega_k}(x)\}_{k=1}^K$ via the truncated Karhunen--Lo$\grave{\text{e}}$ve expansion \eqref{Sto-Coeff-KLE};
\STATE -- determine their optimal pathwise controls $\{f^*_{h,\omega_k}(x) \approx f^{(I)}_{h,\omega_k}(x) \}_{k=1}^K$ in parallel:
\STATE \hspace*{.6em} \textbf{initialize} $f^{(0)}_{h,\omega_k}=0$, \textbf{then} \textbf{for} $i = 0$ \textbf{to} $I - 1$ \textbf{do}

\STATE \hspace*{1.2em} find $u_{h,\omega_k}^{(i)} \in V_h$ such that $\displaystyle b_k \big(u_{h,\omega_k}^{(i)}, v_h\big) = \big(f_{h,\omega_k}^{(i)}, v_h\big)$ for any $v_h\in V_h$;
\STATE \hspace*{1.2em} find $\xi_{h,\omega_k}^{(i)} \in V_h$ such that $\displaystyle b_k\big(\xi_{h,\omega_k}^{(i)}, \eta_h\big) = \big(u_{h,\omega_k}^{(i)}-U, \eta_h\big)$ for any $\eta_h \in V_h$;
\STATE \hspace*{1.2em} control update via $\displaystyle f^{(i+1)}_{h,\omega_k} = f^{(i)}_{h,\omega_k} - \eta \big(\beta f^{(i)}_{h,\omega_k} + \xi^{(i)}_{h,\omega_k}\big)$ with step size $\eta>0$;
\STATE \hspace*{.6em} \textbf{end for} 

\STATE \textit{\% Training Process}
\STATE -- convert $\{a_{h,\omega_k}(x), f^*_{h,\omega_k}(x)\}_{k=1}^K$ into normalized image tensors $\{ A_{\omega_k}, F^*_{\omega_k} \}_{k=1}^K$;
\STATE \textbf{initialize} $\mathcal{N}_{\theta}$, \textbf{then} \textbf{for} $j = 0$ \textbf{to} $J - 1$ \textbf{do}
\STATE -- construct empirical loss over a randomly shuffled mini-batch $\{ A_{\omega_i}, F^*_{\omega_i} \}_{i=1}^B$:
\vspace{-.7em}
\begin{equation*}
	\mathcal{L}( \theta ) = \frac{1}{B} \sum_{i=1}^B \big\| \mathcal{N}_\theta(A_{\omega_i}) - F^*_{\omega_i} \big\|_2^2;
\end{equation*}
\vspace{-1.1em}
\STATE -- update trainable parameters via optimizer with a suitable learning rate $\tau>0$;
\STATE \textbf{end for} 

\STATE \textit{\% Testing Process}
\STATE -- infer pathwise controls for unseen coefficients through a single forward pass:
\vspace{-.85em}
\begin{equation*}
	 f_{h,\omega_k}^{\ms \theta^*}(x) = \sum_{n=1}^N  \left[ F^*_{\omega} \right]_n \phi_n(x)\ \ \textnormal{where}\ \ F^*_{\omega} = \mathcal{N}_{\theta^*}(A_{\omega}).
\end{equation*}
\vspace{-1.2em}
\end{algorithmic}
\label{Alg-RpCtrl-DataDriven}
\end{algorithm}

For the data-driven training of the surrogate models, a large training dataset of input-output pairs $\{ A_{\omega_k}, F^*_{\omega_k} \}_{k=1}^K$ is computed via the traditional Monte Carlo finite element simulations (see Algorithm~\ref{Alg-RpCtrl-DataDriven} for details \cite{ali2017multilevel}), followed by the minimization of the mean squared error between network predictions and target controls
\begin{equation*} 
    \theta^* = \operatorname*{arg\,min}_{\theta\in\mathbb{R}^p} \frac{1}{K} \sum_{k=1}^K \big\| \mathcal{N}_\theta(A_{\omega_k}) - F_{\omega_k}^* \big\|_2^2
\end{equation*}
using a mini-batch gradient descent algorithm. Unlike existing neural operators that typically serve to accelerate the computation of state and adjoint equations \cite{wang2021fast,hwang2022solving}, our trained model $\mathcal{N}_{\theta^*}(A_\omega)$ bypasses the iterative optimization procedure. Owing to its ability to generalize to unseen coefficient realizations (see Section~\ref{Sec-Num-Exps}), the optimal pathwise control can be inferred through a single forward pass within seconds.

While traditional deterministic and stochastic control strategies (see Section~\ref{Sec-Ctrl-Overview} or \autoref{fig-big-picture}) compromise between robustness and adaptivity, our proposed approach unifies the advantages of both and is thus referred to as the robust-pathwise control. The supervised learning procedure is detailed in Algorithm~\ref{Alg-RpCtrl-DataDriven}, in which the trained model achieves superior target-matching performance compared to the deterministic control and drastically reduced inference time compared to the pathwise control.


\section{Physics-informed robust-pathwise control}\label{Sec-RpCtrl-physics-informed}

To eliminate the reliance on high-quality labeled data, this section provides a physics-informed learning approach to train our robust-pathwise control models. The derivation of such an unsupervised approach is detailed, alongside a theoretical study on the neural operator that maps the random coefficient to its corresponding optimal robust stochastic control.

\begin{figure}[t!]
\centering
\includegraphics[width=0.88\textwidth]{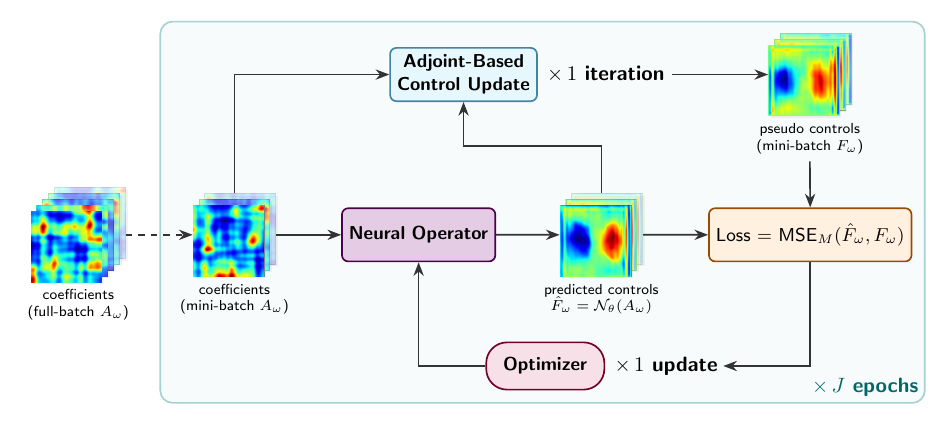}
\vspace{-1.em}
\caption{Unsupervised learning flowchart for our proposed physics-informed robust-pathwise control.}
\label{figure-flowchart-RpCtrl-physics-informed}
\vspace{-0.5em}
\end{figure}

\subsection{Stochastic control via neural operator} 

Given a high-quality training dataset, the data-driven robust-pathwise control \eqref{RpCtrl-data-driven} provides an efficient computational surrogate for our benchmark robust stochastic control problem \eqref{Sto-Ctl-Weak}, enabling rapid inference with satisfactory accuracy across different network architectures (see Section~\ref{Sec-Num-Exps}). However, the substantial cost of data generation remains a bottleneck in practice, underscoring the need for unsupervised learning algorithms.

To eliminate the data dependency, the robust stochastic control problem \eqref{Sto-Ctl-Weak} is resolved by parametrizing the control function through a neural operator, therefore resulting in a finite-dimensional optimization problem over trainable parameters. By including a Tikhonov regularization term for network parameters, the loss function
\begin{equation}\label{Loss-RpCtrl-unSL-continuous}
\mathcal{J} ( \theta ) = \mathbb{E} \left[ \frac12 \| u(x,\omega) - U(x) \|^2_{L^2(D)} + \frac{\beta}{2} \| f_\theta(x,\omega ) \|^2_{L^2(D)} \right] + \frac{\gamma}{2} \| \theta \|_2^2
\end{equation}
is minimized subject to the elliptic problem with uncertain coefficients
\begin{equation}\label{Constraints-RpCtrl-unSL-continuous}
\begin{array}{cl}
\displaystyle - \nabla \cdot \left( a(x,\omega)\nabla u(x,\omega)\right) = f_\theta( x,\omega ),\ \ & \textnormal{in}\ D \otimes\Omega,\\
\displaystyle u(x,\omega)=0,\ \ & \textnormal{on}\ \partial D \otimes \Omega.
\end{array}
\end{equation}
Here, $\gamma > 0$ represents a regularization parameter and the solution mapping \eqref{Solution-Mapping-StoCtrl} is parametrized by a neural operator $f_\theta(x,\omega) = \mathcal{N}_\theta( a(x,\omega) )$ that is assumed to satisfy
\begin{itemize}
\item[(A1)] the mapping $\theta \mapsto f_\theta(\cdot,\omega)$ is continuous in $L^2(D)$ for a.e. $\omega \in \Omega$;
\item[(A2)] $\lVert f_\theta(\cdot,\omega) \rVert_{L^2(D)} \le C(\omega)(1 + \|\theta\|_2)$ with $C(\omega)\in L^2(\Omega)$ for a.e. $\omega \in \Omega$.
\end{itemize}
It is worth noting that the neural operator $\mathcal{N}_\theta( a(x,\omega) )$ establishes a mapping \eqref{Solution-Mapping-StoCtrl} between infinite-dimensional spaces, while its discrete counterpart $\mathcal{N}_\theta( A_\omega )$ defines a mapping \eqref{RpCtrl-data-driven} between finite-dimensional vector or tensor fields.

Note that under the assumption \eqref{Assumption-Coeff} on random coefficients, the state equation \eqref{Constraints-RpCtrl-unSL-continuous} admits a unique weak solution $u\in L^2(\Omega; H_0^1(D))$ for any static model parameter $\theta\in\mathbb{R}^p$. Moreover, an application of the Lax-Milgram lemma and the Poincar\'e inequality \cite{brenner2007mathematical,hou2011finite} yields 
\begin{equation}
	\lVert u \rVert_{L^2(\Omega; H_0^1(D))} \leq \frac{C}{a_{\textnormal{min}}} \lVert f_\theta \rVert_{L^2(\Omega; L^2(D))}
	\label{Priori-Estimate}
\end{equation}
for some generic constant $C>0$. Then we detail the existence proof for the solution to our physics-informed robust-pathwise control problem, following the conventional framework \cite{lions1971optimal} but within a finite-dimensional parameter space.

\begin{theorem}
Under assumptions (A1) and (A2), there exist model parameters $\theta^* \in \mathbb{R}^p$ that minimize the objective function \eqref{Loss-RpCtrl-unSL-continuous} under constraint equations \eqref{Constraints-RpCtrl-unSL-continuous}.
\end{theorem}
\begin{proof}
Note that the Tikhonov regularization term incorporated in \eqref{Loss-RpCtrl-unSL-continuous} implies the coercivity of the objective function \cite{lions1971optimal}, that is,
\begin{equation*}
	\mathcal{J}(\theta) \ge \frac{\gamma}{2} \lVert \theta \rVert_2^2.
\end{equation*}
As a result, any minimizing sequence $\{\theta^{(j)} \in \mathbb{R}^p \}_{j=1}^\infty$ such that
\begin{equation*}
	\lim_{j \to \infty} \mathcal{J}(\theta^{(j)}) = \inf_{\theta \in \mathbb{R}^p} \mathcal{J}(\theta)
\end{equation*}
is uniformly bounded \cite{lions1971optimal}, which suggests the existence of a convergent subsequence $\{\theta^{(j')} \}_{j'=1}^\infty$ with some limit $\theta^* \in \mathbb{R}^p$, namely, $\theta^{(j')} \to \theta^*$ in $\mathbb{R}^p$ as $j'\to\infty$. We denote by $\displaystyle K = \sup \lVert \theta^{(j')} \rVert_2 < \infty$ the uniform bound of the subsequence.

Accordingly, the continuity assumption (A1) of neural operator implies that
\begin{equation*}
	\lim_{j' \to \infty} \lVert f_{\theta^{(j')}}(\cdot, \omega) - f_{\theta^*}(\cdot, \omega) \rVert_{L^2(D)} = 0 \quad \textnormal{for a.e. } \omega \in \Omega.
\end{equation*}
On the other hand, assumption (A2) leads to an integrable dominating envelope
\begin{equation*}
	\lVert f_{\theta^{(j')}}(\cdot, \omega)  - f_{\theta^*}(\cdot, \omega)  \rVert_{L^2(D)}^2 \leq 4 C(\omega)^2 (1 + K)^2,
\end{equation*}
which enables the use of the Lebesgue dominated convergence theorem to arrive at
\begin{equation*}
	\lim_{j' \rightarrow \infty} \mathbb{E}\left[ \|f_{\theta^{(j')}} - f_{\theta^*}\|_{L^2(D)}^2 \right] = \mathbb{E}\left[ \lim_{j' \rightarrow \infty}  \|f_{\theta^{(j')}} - f_{\theta^*}\|_{L^2(D)}^2 \right] = 0.
\end{equation*}

Owing to the linearity of the governing equations \eqref{Constraints-RpCtrl-unSL-continuous} and the estimation \eqref{Priori-Estimate}, the corresponding state difference vanishes in the limit, that is,
\begin{equation*}
	\lim_{j' \rightarrow \infty} \lVert u^{(j')} - u^* \rVert_{L^2(\Omega; H_0^1(D))} \leq \lim_{j' \rightarrow \infty} \frac{C}{a_{\textnormal{min}}} \lVert f_{\theta^{(j')}} - f_{\theta^*} \rVert_{L^2(\Omega; L^2(D))} = 0,
\end{equation*}
where $u^{(j')}$ and $u^*$ denote weak solutions to \eqref{Constraints-RpCtrl-unSL-continuous} with source terms $f_{\theta^{(j')}}$ and $f_{\theta^*}$, respectively. Consequently, the strong convergence of model parameters, control and state functions allow a direct pass to the limit in the loss function \eqref{Loss-RpCtrl-unSL-continuous}, namely,
\begin{equation*}
	\mathcal{J}(\theta^*) = \lim_{j' \to \infty} \mathcal{J}(\theta^{(j')}) = \inf_{\theta \in \mathbb{R}^p} \mathcal{J}(\theta),
\end{equation*}
which confirms the existence of trainable parameters $\theta^*\in\mathbb{R}^p$ that minimize the loss function \eqref{Loss-RpCtrl-unSL-continuous} subject to the constrained equations \eqref{Constraints-RpCtrl-unSL-continuous}.
\end{proof}

\subsection{Unsupervised neural operator learning} \label{Sec-RpCtrl-physics-informed-algorithm}

Upon applying the standard Monte Carlo method over coefficient realizations $\{a_{h,\omega_k}(x)\}_{k=1}^K$ alongside a Galerkin projection onto our finite element space $V_h = \text{span}\{\phi_1, \cdots, \phi_N\} \!\subset\! H_0^1(D)$, the robust stochastic control problem \eqref{Sto-Ctl-Weak} could be computationally recast as an unsupervised task. To be precise, the discrete coefficients of our proposed robust-pathwise control function are generated through a neural operator for each sample realization
\begin{equation*}
	f_{h,\omega_k}^{\ms \theta}(x) = \sum_{n=1}^N  \left[ \mathcal{N}_\theta( A_{\omega_k} ) \right]_n \phi_n(x).
\end{equation*}
Here, $\mathcal{N}_\theta( \cdot )$ is the neural operator introduced in \eqref{RpCtrl-data-driven} for supervised learning, whose output is flattened into a vector corresponding to finite element basis functions.

The inclusion of a Tikhonov regularization term for trainable parameters results in minimizing the objective function
\begin{equation}
	\mathcal{J} (\theta) = \frac{1}{K}\sum_{k=1}^K \left( \frac12 \| u_{h,\omega_k} - U \|^2_{L^2(D)} + \frac{\beta}{2} \| f_{h,\omega_k}^{\ms \theta} \|^2_{L^2(D)} \right) + \frac{\gamma}{2}\|\theta\|^2
\end{equation}
subject to the following weak elliptic constraints for $1\leq k\leq K$
\begin{equation}
	b_k( u_{h,\omega_k},  v_h ) = (f_{h,\omega_k}^{\ms \theta}, v_h), \ \ \ \textnormal{for any}\ v_h \in V_h.
	\label{RpCtl-StateEqn-Weak}
\end{equation}

Next, by introducing the test functions of \eqref{RpCtl-StateEqn-Weak} as Lagrange multipliers for each governing equation, we arrive at the unconstrained Lagrangian functional 
\begin{equation*}
	\mathcal{F} ( \{u_{h,\omega_k}\}_{k=1}^K, \{\xi_{h,\omega_k}\}_{k=1}^K, \theta ) = \mathcal{J}(\theta) - \frac{1}{K}\sum_{k=1}^K \left( b_k( u_{h,\omega_k},  \xi_{h,\omega_k} ) - (f_{h,\omega_k}^{\ms \theta}, \xi_{h,\omega_k}) \right)
\end{equation*}
where the Lagrange multiplier $\xi_{h,\omega_k}\in V_h$ is required to satisfy the adjoint equation 
\begin{equation}
	b_k( \xi_{h,\omega_k},  \eta_h ) = (u_{h,\omega_k} - U, \eta_h), \ \ \ \textnormal{for any}\ \eta_h \in V_h.
	\label{RpCtl-AdjointEqn-Weak}
\end{equation}
On the other hand, evaluating the gradient of Lagrangian functional with respect to model parameters at the solutions of state \eqref{RpCtl-StateEqn-Weak} and adjoint \eqref{RpCtl-AdjointEqn-Weak} equations gives
\begin{equation*}
\begin{aligned}
	\nabla_\theta \mathcal{F} & = \frac{1}{K}\sum_{k=1}^K \big( \beta f_{h,\omega_k}^{\ms\theta} + \xi_{h,\omega_k}, \nabla_\theta f_{h,\omega_k}^{\ms\theta} \big) + \gamma \theta \\
	& = \frac{1}{K}\sum_{k=1}^K \big( \beta f_{h,\omega_k}^{\ms\theta} + \xi_{h,\omega_k}, \sum_{n=1}^N  \nabla_\theta \left[ \mathcal{N}_\theta( A_{\omega_k} ) \right]_n \phi_n(x) \big) + \gamma \theta \\
	& = \frac{1}{K}\sum_{k=1}^K \big( \nabla_\theta \mathcal{N}_\theta( A_{\omega_k} )  \big)^T M \big( \beta \mathcal{N}_\theta( A_{\omega_k} ) + \Xi_{\omega_k} \big) + \gamma \theta.
\end{aligned}	
\end{equation*}
where the final equality recasts the gradient into a matrix-vector form by expanding the adjoint function in terms of basis functions and introducing the mass matrix
\begin{equation*}
	\xi_{h,\omega_k}(x) = \sum_{n=1}^N \left[ \Xi_{\omega_k} \right]_n \phi_n(x) \quad \textnormal{and} \quad  M = \left[ \int_D \phi_i(x) \phi_j(x) \, dx \right]_{i,j=1}^N.
\end{equation*}

Specifically, the $j$-th iteration step proceeds by first solving the state \eqref{RpCtl-StateEqn-Weak} and adjoint \eqref{RpCtl-AdjointEqn-Weak} equations with control surrogate $\mathcal{N}_{\theta^{(j)}}( A_{\omega_k} ) \in\mathbb{R}^N$. Once the coefficient vector $\Xi_{\omega_k}^{(j)} \in \mathbb{R}^N$ is obtained across all sample realizations, the trainable parameters are updated according to the full-batch gradient descent with a learning rate $\tau>0$
\begin{equation}
	\theta^{(j+1)} = (1-\tau \gamma) \theta^{(j)} - \frac{\tau}{K}\sum_{k=1}^K \big( \nabla_\theta \mathcal{N}_{\theta^{(j)}}( A_{\omega_k} )  \big)^T M \big( \beta \mathcal{N}_{\theta^{(j)}}( A_{\omega_k} ) + \Xi_{\omega_k}^{(j)} \big).
	\label{RpCtl-Parameter-Update}
\end{equation}
To facilitate the practical implementation, we introduce a pseudo control target with a step size parameter $\eta>0$, which is detached from the computational graph
\begin{equation*}
	F_{\omega_k}^{(j+1)} = \mathcal{N}_{\theta^{(j)}}( A_{\omega_k} )  - \eta \big( \beta \mathcal{N}_{\theta^{(j)}}( A_{\omega_k} ) + \Xi_{\omega_k}^{(j)} \big).
\end{equation*}
This allows the update rule \eqref{RpCtl-Parameter-Update} to be recast as a single backpropagation step that aims at minimizing the weakly supervised loss with a learning rate $\tau/\eta$
\begin{equation}
	\mathcal{L}(\theta) = \frac{1}{2K} \sum_{k=1}^K \lVert  \mathcal{N}_\theta(A_{\omega_k}) - F_{\omega_k}^{(j+1)} \rVert_M^2,
	\label{RpCtl-Loss-Unsupervised}
\end{equation}
under weight decay \cite{loshchilov2017decoupled}. The $M$-weighted norm is defined as $\Vert{} v \Vert{}_M^2 = v^\top M v$ for any $v \in \mathbb{R}^N$ and the gradient of our Tikhonov regularization term \eqref{RpCtl-Parameter-Update} is conducted via the weight decay module of deep learning optimizers \cite{loshchilov2017decoupled}.

\begin{remark}
Calculating the gradient of our weakly supervised loss \eqref{RpCtl-Loss-Unsupervised} at the $j$-th iteration via the chain rule yields
\begin{equation*}
	\nabla_\theta \mathcal{L}( \theta^{(j)} ) = \frac{1}{K} \sum_{k=1}^K \big( \nabla_\theta \mathcal{N}_{\theta^{(j)}}(A_{\omega_k}) \big)^T M \big( \mathcal{N}_{\theta^{(j)}}(A_{\omega_k}) - F_{\omega_k}^{(j+1)} \big).
\end{equation*}
Then the substitution of our pseudo control target $F_{\omega_k}^{(j+1)}$ induces a cancellation of the detached network output $\mathcal{N}_{\theta^{(j)}}(A_{\omega_k})$, thereby resulting in
\begin{equation*}
	\nabla_\theta \mathcal{L}(\theta^{(j)}) = \frac{\eta}{K} \sum_{k=1}^K \big( \nabla_\theta \mathcal{N}_{\theta^{(j)}}(A_{\omega_k}) \big)^T M  \big( \beta \mathcal{N}_{\theta^{(j)}}(A_{\omega_k}) + \Xi_{\omega_k}^{(j)} \big) 
\end{equation*}
that recovers exactly the adjoint-based parameter update \eqref{RpCtl-Parameter-Update} under weight decay.
\end{remark}

It is worth noting that while the original robust stochastic control problem \eqref{Sto-Ctl-Weak} is convex, our parametrization of control function using neural operators \cite{ronneberger2015u,jia2025novel,lu2021learning} renders the weakly supervised learning task \eqref{RpCtl-Loss-Unsupervised} highly non-convex. Then, the full-batch gradient descent approach \eqref{RpCtl-Parameter-Update} becomes both computationally expensive and prone to trapping in local minima or saddle points, thereby motivating the adoption of a mini-batch training loss 
\begin{equation}
	\mathcal{L}(\theta) = \frac{1}{2B} \sum_{i=1}^B \lVert  \mathcal{N}_\theta(A_{\omega_i}) - F_{\omega_i}^{(j+1)} \rVert_M^2,
	\label{RpCtl-Loss-Unsupervised-MiniBatch}
\end{equation}
over $B \ll K$ sample realizations instead of \eqref{RpCtl-Loss-Unsupervised} in practice. Implementation details are illustrated in Algorithm~\ref{Alg-RpCtrl-PhysicsInformed}, followed by a concise discussion on the reduction in computational cost achieved via the generalization capability of neural operators.

\begin{algorithm}[t!]
\caption{Physics-Informed Robust-Pathwise Control}
\begin{algorithmic}
\STATE \textbf{Input:} step sizes $\eta,\tau>0$, mini-batch data size $B$ and maximum epoch count $J$.
\STATE \textbf{Output:} physics-informed robust-pathwise control $\mathcal{N}_{\theta^*}$.

\STATE \textit{\% Data Generation}
\STATE -- compute $\{ \lambda_m^h, \varphi^h_m(x) \}_{m\geq 1}$ via the generalized eigenvalue problem \eqref{KLE-Generalized-Eigenprob};
\STATE -- generate $\{a_{h,\omega_k}(x)\}_{k=1}^K$ via \eqref{Sto-Coeff-KLE} to construct normalized tensors $\{ A_{\omega_k} \}_{k=1}^K$;

\STATE \textit{\% Training Process}

\STATE \textbf{initialize} $\mathcal{N}_{\theta}$, \textbf{then} \textbf{for} $j = 0$ \textbf{to} $J - 1$ \textbf{do}

\STATE -- compute pseudo controls $\{F_{\omega_i} \}_{i=1}^B$ for mini-batch inputs $\{ A_{\omega_i} \}_{i=1}^B$ in parallel:
\STATE \hspace*{.6em} forward pass of neural operator $\mathcal{N}_\theta(A_{\omega_i})$ to get predicted controls $f_{h,\omega_i}^{\ms\theta}(x)$;
\STATE \hspace*{.6em} find $u_{h,\omega_i} \in V_h$ such that $\displaystyle b_i (u_{h,\omega_i}, v_h) = (f_{h,\omega_i}^{\ms \theta}, v_h )$ for any $v_h\in V_h$;
\STATE \hspace*{.6em} find $\xi_{h,\omega_i} \in V_h$ such that $\displaystyle b_i (\xi_{h,\omega_i}, \eta_h ) = (u_{h,\omega_i}-U, \eta_h)$ for any $\eta_h \in V_h$;
\STATE \hspace*{.6em} control update via $\displaystyle F_{\omega_i} = \mathcal{N}_{\theta}( A_{\omega_i} )  - \eta ( \beta \mathcal{N}_{\theta}( A_{\omega_i} ) + \Xi_{\omega_i} )$ with step size $\eta>0$;

\STATE -- detach pseudo controls from computational graph and construct empirical loss:
\vspace{-.7em}
\begin{equation*}
	\mathcal{L}( \theta ) = \frac{1}{2B} \sum_{i=1}^B \big\| \mathcal{N}_\theta(A_{\omega_i}) - F_{\omega_i} \big\|_M^2;
\end{equation*}
\vspace{-1.1em}
\STATE -- update parameters via a single step of gradient descent with weight decay;
\STATE \textbf{end for} 

\STATE \textit{\% Testing Process}
\STATE -- infer pathwise controls for unseen coefficients through a single forward pass:
\vspace{-.85em}
\begin{equation*}
	 f_{h,\omega}^{\ms {\theta^*}}(x) = \sum_{n=1}^N  \left[ \mathcal{N}_{\theta^*}( A_{\omega} ) \right]_n \phi_n(x).
\end{equation*}
\vspace{-1.2em}
\STATE \hspace*{.5em} (identical notation from Algorithm~\ref{Alg-RpCtrl-DataDriven} is reused and will be clarified as needed)
\end{algorithmic}
\label{Alg-RpCtrl-PhysicsInformed}
\end{algorithm}

Analogous to the data-driven approach outlined in Algorithm~\ref{Alg-RpCtrl-DataDriven}, the dominant computational overhead within our physics-informed alternative is attributed to the repeated execution of adjoint-based control updates (namely, the sequential solution of state and adjoint equations, followed by the control update). However, Algorithm ~\ref{Alg-RpCtrl-PhysicsInformed} leverages mini-batch training to efficiently generalize across the entire dataset at each iteration, bypassing the need to execute control updates from scratch and store pathwise control functions. This significantly contrasts with the data-driven method presented in Algorithm~\ref{Alg-RpCtrl-DataDriven} as well as the robust deterministic control, which require a complete iteration of control updates for each coefficient realization. 


\section{Numerical experiments}\label{Sec-Num-Exps}

This section presents the computational evaluation of our data-driven and physics-informed robust-pathwise control strategies (see Algorithm~\ref{Alg-RpCtrl-DataDriven} and \ref{Alg-RpCtrl-PhysicsInformed}), applied to minimize the tracking objective \eqref{Sto-Ctl-Objective-Strong} constrained by a second-order elliptic problem \eqref{Sto-Ctl-Constraint-Strong} with stochastic coefficients. The benchmark robust deterministic control is employed as our baseline for comparison, evaluated in terms of tracking error and control effort across unseen coefficient realizations.

To be specific, the deterministic target function in our objective functional \eqref{Sto-Ctl-Objective-Strong} is defined as $U(x_1,x_2)=\sin(2\pi x_1)(\cos(2\pi x_2)-1)$ on a two-dimensional unit square domain $D=[0,1]^2$, along with regularization constant $\beta=10^{-10}$. Spatial discretization across all experiments is conducted using continuous piecewise linear elements on a uniform triangulation $\mathcal{T}_h$ with mesh width $h=2^{-5}$ (unless otherwise specified), resulting in the finite element space $V_h=\textnormal{span}\{\phi_1(x),\cdots,\phi_n(x)\}$. In particular, the log-normal coefficient in \eqref{Sto-Ctl-Constraint-Strong} exhibits limited regularity as detailed in section~\ref{Sec-NumExps-RpCtrl-SL-train}, which results in a high-dimensional representation \eqref{Sto-Coeff-KLE} when approximated through the Karhunen--Lo$\grave{\text{e}}$ve expansion. Accordingly, the Monte Carlo method is utilized to lessen the curse of dimensionality, together with the parametrization of our solution operator via neural network surrogate models.

Model training is conducted through PyTorch on the Nvidia GeForce RTX 4090 GPU cards, with hyperparameter configurations summarized in \autoref{tab:app_arch} for brevity. To ensure computational reproducibility, source codes will be made publicly accessible at \url{https://github.com/q1sun/RpCtrl_StoPDEs} upon the acceptance.

\subsection{Data-driven robust-pathwise control} 

To validate the proposed data-driven method, the following sections first detail the generation of input-output data pairs via the optimal pathwise control strategy, followed by the evaluation of trained model on unseen coefficient realizations against the robust deterministic control. An ablation study on the neural operator architectures concludes this section, alongside numerical experiments extending the proposed method to rougher coefficients.

\subsubsection{Training performance} \label{Sec-NumExps-RpCtrl-SL-train}

To construct the training dataset, we consider the covariance function \eqref{Sto-Coeff-Cov} characterized by $\sigma=10$, $\ell=0.2$, and $\lVert\cdot\rVert=\lVert\cdot\rVert_{L^1(D)}$ as an illustrative example. Following a uniform triangulation of the domain $D=[0,1]^2$ with mesh size $h=2^{-5}$, \autoref{figure-KLE-Eigvalue} shows the largest 40 eigenvalues of \eqref{KLE-Generalized-Eigenprob} as well as their corresponding cumulative energy ratios \eqref{KLE-Eigenprob-Ratio}. This example demonstrates that retaining an energy ratio of $95\%$ for $g(x,\omega)$ requires at least $M\geq 166$ terms in \eqref{KLE-Log-Coeff} and \eqref{Sto-Coeff-KLE}, resulting in high-dimensional random inputs for problem \eqref{Sto-Ctl-Weak}.

\begin{figure}[H]
\centering
\includegraphics[width=0.9\textwidth]{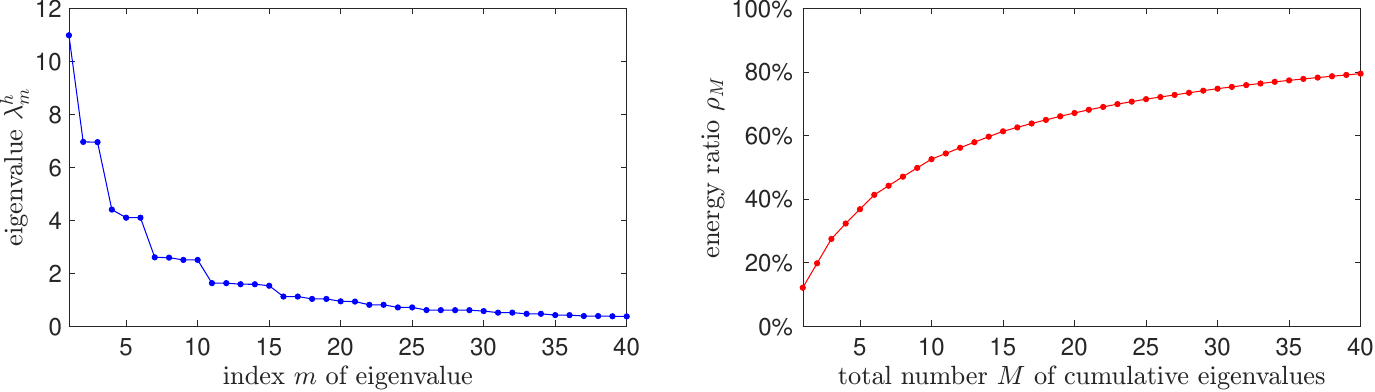}
\vspace{-0.3cm}
\caption{Left to right: the largest 40 eigenvalues and their cumulative energy ratios for \eqref{KLE-Generalized-Eigenprob}.}
\label{figure-KLE-Eigvalue}
\end{figure}

Specifically, the numerical coefficient on a $33\times 33$ spatial grid takes the form
\begin{equation*}
	a_{h,\omega}(x) = \sum_{n=1}^{1089} \textnormal{exp}\left( \sum_{m=1}^{166} c_{mn} \sqrt{\lambda_m^h} \theta_m(\omega)  \right) \phi_n(x),	
\end{equation*}
and the first column of \autoref{figure-train-results-SL} shows its sample realizations with random variables drawn independently and identically from the standard normal distribution $\mathcal{N}(0,1)$. A uniform distribution over $[-\sqrt{3}, \sqrt{3}]$ offers an alternative with zero mean and unit variance, though it is omitted here. Accordingly, the numerical approximation to its optimal pathwise control $f_{h,\omega}^*(x)$ is computed through Algorithm~\ref{Alg-RpCtrl-DataDriven} as displayed in the second column of \autoref{figure-train-results-SL}, executing $I=3\times 10^3$ iterations across $K=5\times 10^3$ samples in parallel. Therefore, we construct a training dataset of input-output pairs $\{ a_{h,\omega_k}(x), f_{h,\omega_k}^*(x) \}_{k=1}^K$, which can be expressed as $\{ A_{\omega_k}, F_{\omega_k}^* \}_{k=1}^K$ by arranging their finite element coefficients onto the $H\times W=33\times 33$ mesh grid.

\begin{figure}[t!]
\centering
\includegraphics[width=0.7\textwidth]{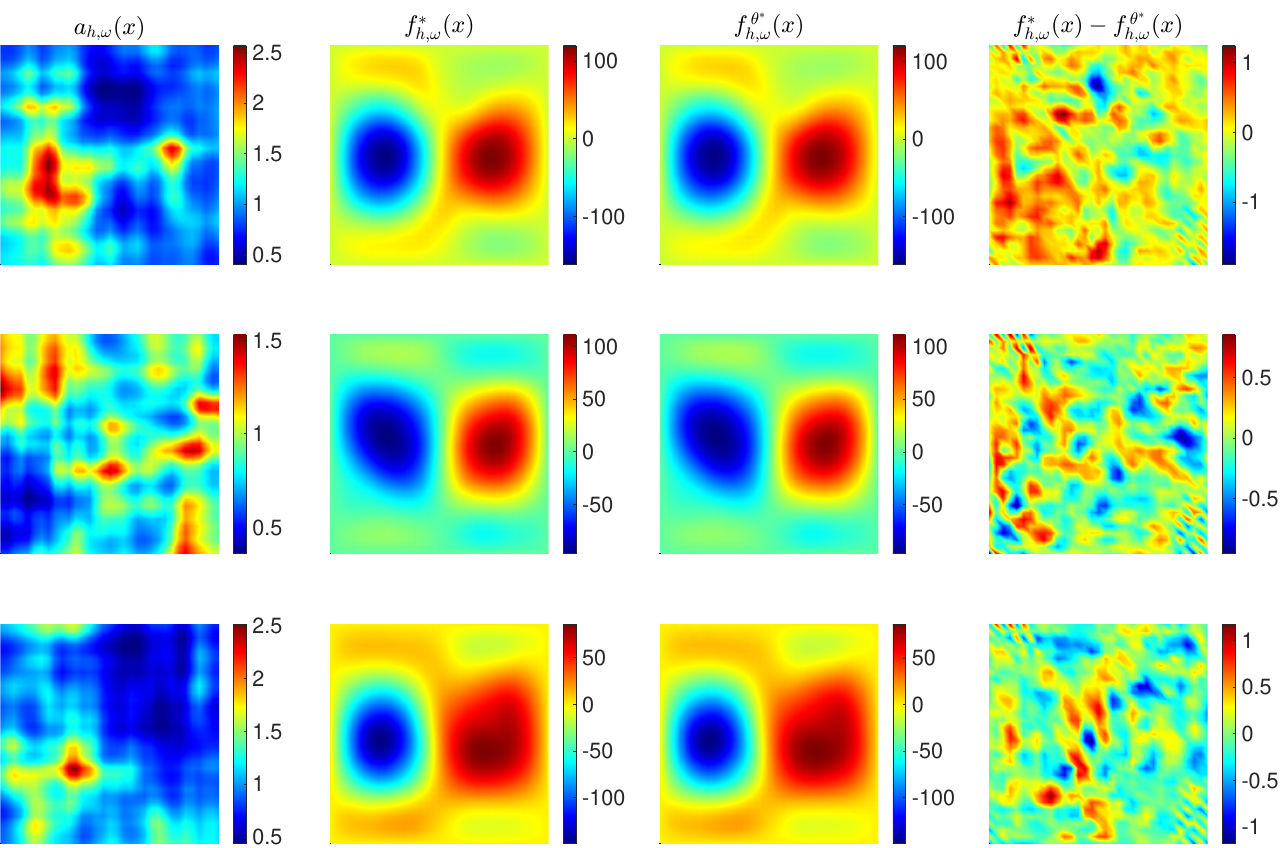}
\vspace{-0.1cm}
\caption{Left to right: coefficient realizations (network input), optimal pathwise controls (target output), data-driven robust-pathwise controls (network prediction), and pointwise error profiles.}
\label{figure-train-results-SL}
\end{figure}

\begin{figure}[t!]
\centering
\includegraphics[width=0.7\textwidth]{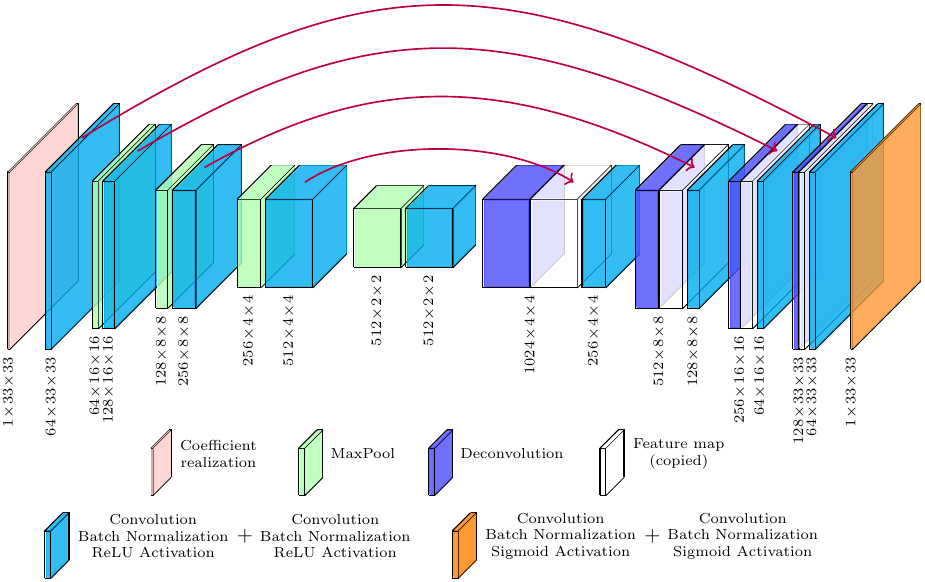}
\vspace*{-0.2cm}
\caption{\textnormal{U-Net architecture for our robust-pathwise control in section~\ref{Sec-Num-Exps}, with feature concatenation indicated by arrows and feature maps color-coded according to layer types.}}
\label{figure-Unet-SL}
\end{figure}

\begin{figure}[t!]
\centering
\includegraphics[width=0.6\textwidth]{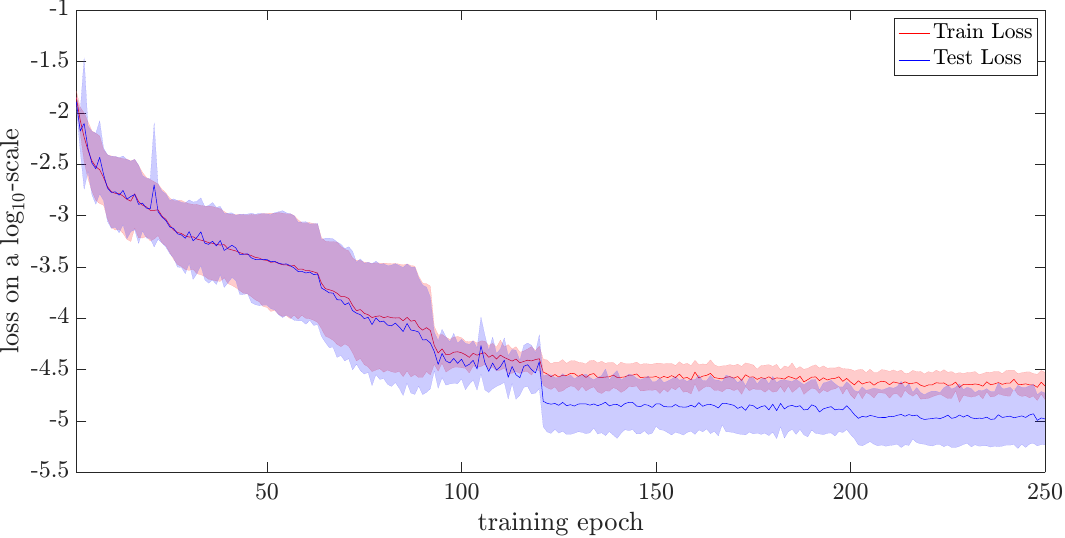}
\vspace{-0.3cm}
\caption{Supervised learning curves (mean $\pm$ standard deviation) over 6 independent runs.}
\label{figure-train-learn-curves-SL}
\vspace{-0.3cm}
\end{figure}

Following the supervised learning paradigm depicted in \autoref{figure-flowchart-RpCtrl-data-driven}, the U-Net \cite{ronneberger2015u} architecture with approximately 13 million trainable parameters is deployed to learn the solution mapping \eqref{Solution-Mapping-StoCtrl-Pathwise-Img2Img} \cite{jia2025novel}. Architected entirely via convolutional layers, U-Net adopts an encoder to extract latent features from coefficient realizations, followed by a decoder that reconstructs the optimal pathwise controls as illustrated in \autoref{figure-Unet-SL}. Training of the surrogate model on the labeled dataset $\{ A_{\omega_k}, F_{\omega_k}^* \}_{k=1}^K$ is executed over 250 epochs utilizing a batch size of $10^2$. Parameter update is performed through the AdamW optimizer \cite{loshchilov2017decoupled} with default settings, starting from a learning rate of $1\times 10^{-3}$ that is divided by 10 at the 120th and 200th epochs.

To account for the training stability, we conduct six independent runs and show their learning curves (mean $\pm$ standard deviation) in \autoref{figure-train-learn-curves-SL}, achieving an average training loss of $1.95 \times 10^{-5}$. We retain the model that attains the minimum training loss of $1.46 \times 10^{-5}$ for inference, which demonstrates satisfactory alignment with the target output across distinct sample realizations as depicted in \autoref{figure-train-results-SL}.

\subsubsection{Generalization ability} 

Note that our trainable parameters are robust to the uncertainty in stochastic coefficients, while the model prediction can adapt to individual coefficient realizations. To evaluate the generalization performance of our learned surrogate model, a testing dataset $\{A_{\omega_\ell}, F_{\omega_\ell}^* \}_{\ell=1}^L$ is independently generated for validation, comprising $L = 1 \times 10^3$ unseen coefficient realizations as well as their numerical optimal pathwise controls.

To provide a performance baseline, the robust deterministic control $\hat{f}_h^*(x)$ \cite{sun2020distributed} is numerically computed via the standard Monte Carlo finite element method, utilizing the same coefficient realizations and iterative settings established in section~\ref{Sec-NumExps-RpCtrl-SL-train}. A quantitative comparison is given in \autoref{Table-TestPerformance-RpCtrl-data-driven}, evaluated in terms of the averaged state tracking error and control effort across the testing dataset for a typical run, that is,
\begin{equation*}
	\frac{1}{L} \sum_{\ell=1}^L \lVert v_{h,\omega_\ell}(x) - U(x) \rVert_{L^2(D)} \ \ \ \textnormal{and}\ \ \  \frac{1}{L} \sum_{\ell=1}^L \lVert g_{h,\omega_\ell}(x) \rVert_{L^2(D)},
\end{equation*}
where $v_{h,\omega_\ell}(x)$ denotes the state function driven by the control function $g_{h,\omega_\ell}(x)$ for the coefficient realization $a_{h,\omega_\ell}(x)$. Here, the control function $g_{h,\omega_\ell}(x)$ is instantiated as either the robust deterministic control, our proposed data-driven robust-pathwise control, or the pathwise control. As demonstrated in \autoref{Table-TestPerformance-RpCtrl-data-driven}, our method achieves a markedly lower tracking error than the commonly used robust deterministic baseline at a negligible inference time and removes the iterative computational burden of the pathwise control without compromising target alignment.

\begin{table}[b!]
\label{Table-TestPerformance-RpCtrl-data-driven}
\vspace{-0.1cm}
\centering
\renewcommand{\arraystretch}{1.4}
\caption{Generalization performance of robust deterministic control $\hat{f}^*_h(x)$, data-driven robust-pathwise control $f^{\ms\theta^*}_{h,\omega}(x)$, and pathwise control $f^*_{h,\omega}(x)$, compared via tracking error and runtime.}
\vspace{-0.1cm}
\label{Table-Runtime-Testdata-KLE} 
    \begin{tabular}{ | c | c || *{3}{>{\centering\arraybackslash}p{20.0mm}|} } 
    \hline
    \multicolumn{2}{|>{\columncolor{kugray5}}c||}{Generalization Test} & \multirow{2}{*}{$\hat{f}^*_h(x)$} & \multirow{2}{*}{$f^{\ms\theta^*}_{h,\omega}(x)$} & \multirow{2}{*}{$f^*_{h,\omega}(x)$} \\
    \multicolumn{2}{|>{\columncolor{kugray5}}c||}{$\ell \!=\! 0.2$, $\sigma \!=\! 10$, $M\!=\!166$} &  &  & \\    
    \hline   					    	
    offline & runtime (s) & 17235.70 & 25868.96 & 0 \\
    \hline   
    \multirow{3}{*}{online} & runtime (s) & 0  & 0.62 & 4594.66  \\
    \cline{2-5} 
    	& tracking error & $1.35\times10^{-1}$  & $3.29\times10^{-2}$ & $3.26\times10^{-2}$  \\
    	\cline{2-5} 
    	& control effort & 45.29 & 48.52 & 48.48  \\    
    \hline   					    	   							                                                                    
  \end{tabular} 
\end{table}

Additionally, empirical distributions of the tracking error and control effort over individual realizations in the testing dataset are compared in \autoref{Fig-TestPerformance-RpCtrl-data-driven}. To provide a visual comparison, \autoref{figure-Adaptivity-RpCtrl-data-driven} illustrates two representative scenarios where the robust deterministic baseline exhibits adequate and degraded tracking performance. Owing to its inherent adaptivity and generalization capability to stochastic coefficients, our trained surrogate model achieves superior target matching in both instances.

\begin{figure}[t!]
\centering
\includegraphics[width=0.75\textwidth]{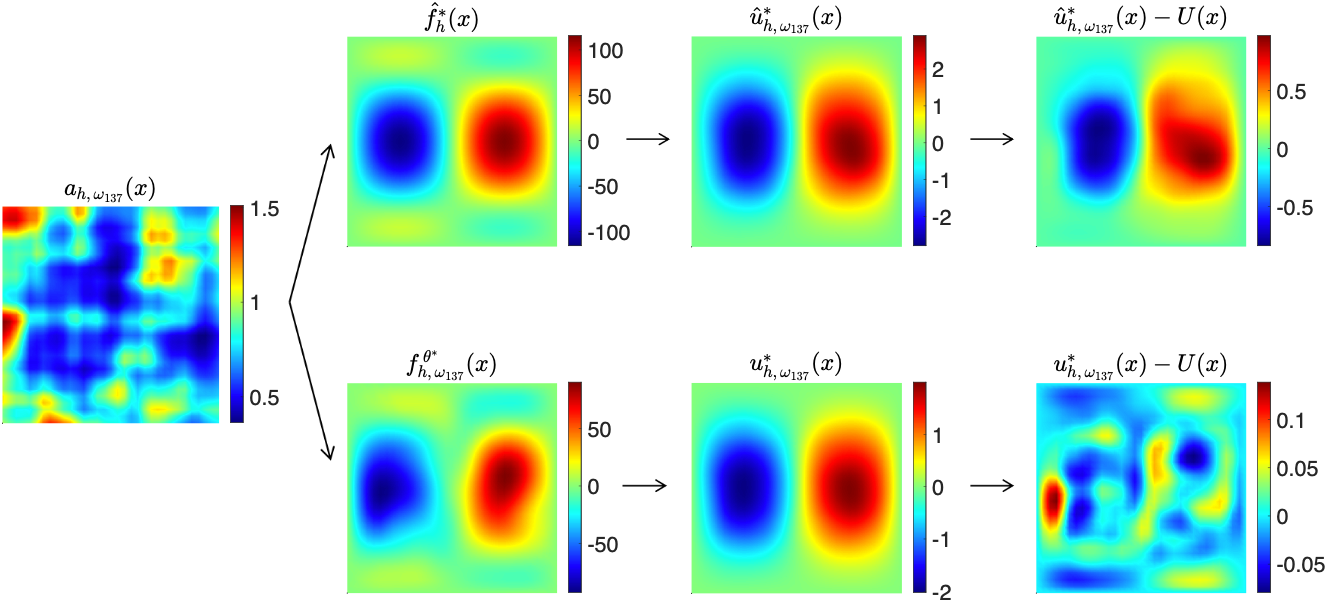}

\vspace{0.2cm}

\includegraphics[width=0.75\textwidth]{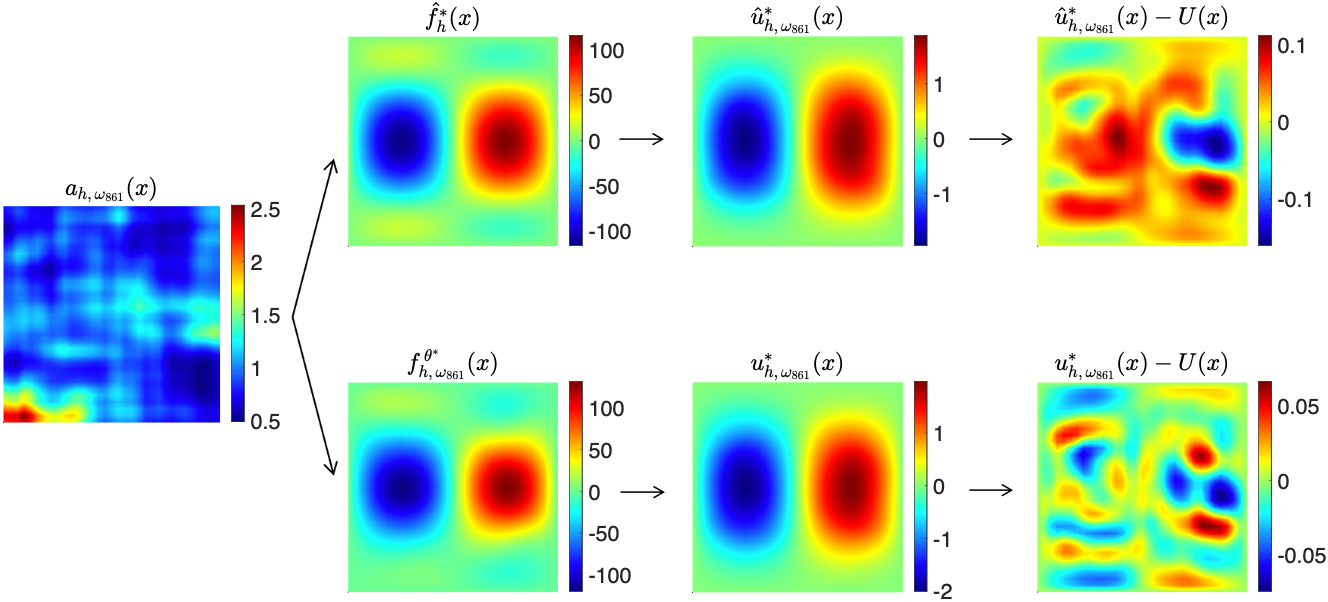}
\vspace{-0.2cm}
\caption{Application of robust deterministic control $\hat{f}^*_h(x)$ and data-driven robust-pathwise control $f^{\ms \theta^*}_{h,\omega}(x)$ to unseen coefficient realizations, assessed in terms of state tracking errors.}
\label{figure-Adaptivity-RpCtrl-data-driven}
\end{figure}

\begin{figure}[t!]
\centering
\includegraphics[width=0.9\textwidth]{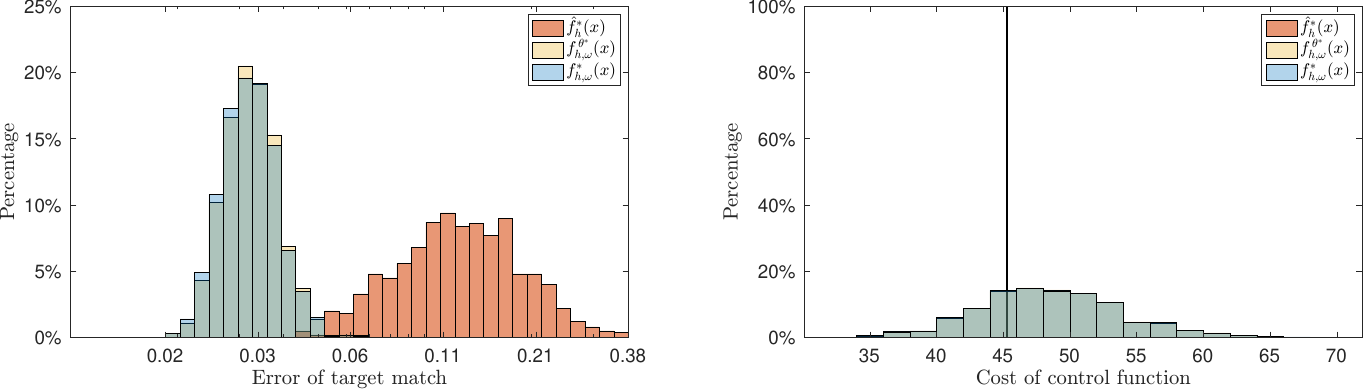}
\vspace{-0.3cm}
\caption{Empirical distributions of tracking error and control effort over individual realizations. }
\label{Fig-TestPerformance-RpCtrl-data-driven}
\vspace{-0.3cm}
\end{figure}

\subsubsection{Architectural ablation of neural operators} 

While previous numerical experiments adopt a U-Net model \cite{ronneberger2015u,jia2025novel,azizzadenesheli2024neural}, our proposed data-driven training algorithm~\ref{Alg-RpCtrl-DataDriven} seamlessly accommodates alternative neural operators, acknowledging that distinct architectures exhibit varying approximation bounds, generalization and optimization complexity. To investigate the effect of architecture design on tracking performance, this section presents an ablation study comparing our U-Net surrogate model against the deep operator network \cite{lu2021learning} and the Fourier neural operator \cite{li2020fourier}.

All models are trained and evaluated on the identical datasets established in the preceding section, with specific hyperparameter settings deferred to \autoref{tab:app_arch} for simplicity. As reported in \autoref{Table-TestPerformance-RpCtrl-data-driven-DeepONet-FNO}, the U-Net model achieves superior generalization performance compared to the other two prominent neural operators and is therefore utilized for all our subsequent numerical experiments. The performance degradation of the Fourier neural operator can be attributed to the truncation of high-frequency modes, which may limit its ability to resolve the oscillatory coefficient realizations.

\begin{table}[t!]
\label{Table-TestPerformance-RpCtrl-data-driven-DeepONet-FNO}
\centering
\renewcommand{\arraystretch}{1.4}
\caption{Comparison of generalization performance across different neural operator architectures.}
\vspace{-0.1cm}
\setlength{\tabcolsep}{4pt} 
\begin{tabular}{lccccc}
\hline
 & Parameters & Testing Error & Tracking Error & Control Effort \\
\hline
U-Net & 13.39\,$m$ & $1.15 \times 10^{-5}$ & $3.29 \times 10^{-2}$ & $48.52$ \\
DeepONet & 1.61\,$m$ & $9.72 \times 10^{-5}$ &$3.44 \times 10^{-2}$ & $48.50$ \\
FNO      & 9.46\,$m$ & $4.12 \times 10^{-4}$ &$3.54 \times 10^{-2}$ & $48.49$ \\       
\hline
\end{tabular}
\end{table}

\begin{table}[t!]
\label{tab:app_arch}
\centering
\small
\setlength{\tabcolsep}{3pt}
\setlength{\extrarowheight}{3pt}
\caption{List of hyper-parameter configurations for neural operators in our numerical experiments, more details will be made publicly accessible at \url{https://github.com/q1sun/RpCtrl_StoPDEs}.}
\vspace{-0.1cm}
\begin{tabular}{ c c c c c c }
\hline
\textbf{Model} & \textbf{Section} & \textbf{ \makecell[c]{Train/Test \\ Data Size}} & \textbf{Batch Size} & \textbf{\makecell[c]{ Training Epochs \\ (Decay Milestones)}} & \textbf{\makecell[c]{ Learning \\ Rate }} \\
\hline
\makecell[c]{} & \makecell[c]{5.1.1 \\ 5.1.2 \\5.1.3 } & \makecell[c]{$(5k,1k)$} & \makecell[c]{100} & \makecell[c]{\makecell[c]{ 250 \\ $(120,200)$ } } & \makecell[c]{$1\times10^{-3}$} \\
\cline{2-6}
U-Net \cite{ronneberger2015u} & \makecell[c]{5.1.4} & \makecell[c]{$(5k,1k)$} & \makecell[c]{100} & \makecell[c]{\makecell[c]{ 250 \\ $(120,200)$ } } & \makecell[c]{$5\times10^{-3}$} \\
\cline{2-6}
\makecell[c]{} & \makecell[c]{5.2} & \makecell[c]{$(5k,1k)$} & \makecell[c]{250} & \makecell[c]{\makecell[c]{ 200 \\ $(160,180)$ } } & \makecell[c]{$4\times10^{-3}$} \\
\hline
DeepONet \cite{lu2021learning} & \makecell[c]{5.1.3} & \makecell[c]{$(5k,1k)$} & \makecell[c]{100} & \makecell[c]{\makecell[c]{ 1000 \\ $(600,800,900)$ } } & \makecell[c]{$1\times10^{-3}$} \\
\hline
FNO \cite{li2020fourier} & \makecell[c]{5.1.3} & \makecell[c]{$(5k,1k)$} & \makecell[c]{100} & \makecell[c]{\makecell[c]{ 100 \\ $(60,80)$ } } & \makecell[c]{$1\times10^{-3}$} \\
\hline
\end{tabular}
\vspace{-0.3cm}
\end{table}

Despite the discrepancy in generalization performance, all three neural operator surrogate models surpass the baseline robust deterministic control by approximately an order of magnitude, even in our absence of hyperparameter refinement.

\begin{table}[b!]
\vspace*{-0.3cm}
\label{Table-TestPerformance-RpCtrl-data-driven-rougher}
\centering
\renewcommand{\arraystretch}{1.4}
\vspace*{-0.2cm}
\caption{Generalization performance of different control strategies for rougher coefficients.}
\label{Table-Runtime-Testdata-KLE} 
    \begin{tabular}{ | c | c || *{3}{>{\centering\arraybackslash}p{20.0mm}|} } 
    \hline
    \multicolumn{2}{|>{\columncolor{kugray5}}c||}{Generalization Test} & \multirow{2}{*}{$\hat{f}^*_h(x)$} & \multirow{2}{*}{$f^{\ms \theta^*}_{h,\omega}(x)$} & \multirow{2}{*}{$f^*_{h,\omega}(x)$} \\
    \multicolumn{2}{|>{\columncolor{kugray5}}c||}{$\ell \!=\! 0.08$, $\sigma \!=\! 10$, $M\!=\!807$} &  &  & \\    
    \hline   					    	
    offline & runtime (s) & 273952.27 & 334348.67 & 0 \\
    \hline   
    \multirow{3}{*}{online} & runtime (s) & 0  & 2.04 & 72156.52  \\
    \cline{2-5} 
    	& tracking error & $4.48\times10^{-2}$  & $1.66\times10^{-2}$ & $1.66\times10^{-2}$  \\
    	\cline{2-5} 
    	& control effort & 47.39 & 47.81 & 47.82  \\    
    \hline   					    	   						                                                                    
  \end{tabular} 
\end{table}

\subsubsection{Robust-pathwise control for rougher coefficients} 

To conclude our numerical study of data-driven robust-pathwise control (Algorithm~\ref{Alg-RpCtrl-DataDriven}), we examine the covariance function \eqref{Sto-Coeff-Cov} with a reduced correlation length $\ell=0.08$, maintaining $\sigma=10$ and $\lVert\cdot\rVert=\lVert\cdot\rVert_{L^1(D)}$ as defined in section~\ref{Sec-NumExps-RpCtrl-SL-train}. Accordingly, the eigenvalues of problem \eqref{KLE-Generalized-Eigenprob} decay slowly, leading to rougher coefficients with an increased number of truncation terms. To be precise, capturing $95\%$ of the energy necessitates at least $M\geq 807$ in \eqref{Sto-Coeff-KLE}, posing a computational challenge for conventional methods due to the curse of dimensionality.

Employing a mesh size of $h=2^{-6}$ in Algorithm~\ref{Alg-RpCtrl-DataDriven}, we evaluate our data-driven robust-pathwise control against alternative control strategies across $L=1\times 10^3$ test samples and report the generalization performance in \autoref{Table-TestPerformance-RpCtrl-data-driven-rougher} and \autoref{figure-train-results-SL-rougher}. Despite the increased roughness of coefficient realizations, our approach effectively adapts to the underlying input samples, yielding a notable reduction in tracking error relative to the robust deterministic benchmark. Remarkably, the end-to-end inference across $L=1\times 10^3$ unknown input coefficient samples requires only 2.04 seconds, bypassing the adjoint-based iteration without sacrificing accuracy relative to the pathwise control.

\begin{figure}[t!]
\centering
\includegraphics[width=0.85\textwidth]{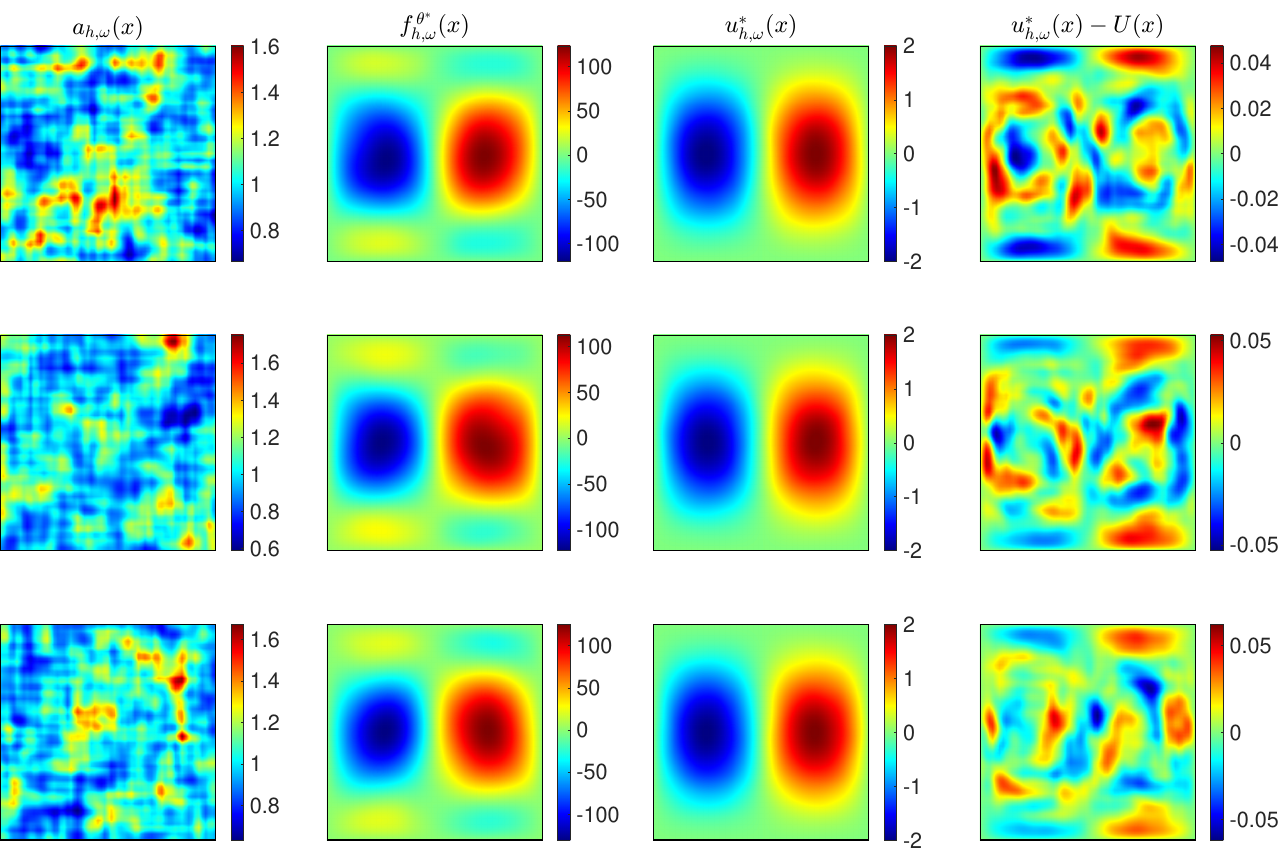}
\vspace{-0.1cm}
\caption{Left to right: coefficient realizations with $\ell = 0.08$, data-driven robust-pathwise control predictions via trained model, state solutions and their tracking error profiles.}
\label{figure-train-results-SL-rougher}
\vspace{-0.3cm}
\end{figure}


\subsection{Physics-informed robust-pathwise control} 

While the effectiveness of data-driven robust-pathwise control is numerically validated in the previous section, the supervised paradigm is hindered by the computationally expensive generation of high-fidelity optimal pathwise control functions. To eliminate such data dependency, this section offers a numerical investigation on the effectiveness and efficiency of our physics-informed approach detailed in Algorithm~\ref{Alg-RpCtrl-PhysicsInformed}.

\subsubsection{Training performance} 

To facilitate a fair comparison with the supervised approach, our physics-informed robust-pathwise control model is trained using the identical configuration presented in section~\ref{Sec-NumExps-RpCtrl-SL-train}. Specifically, a training dataset $\{ A_{\omega_k}\in \mathbb{R}^{33\times 33} \}_{k=1}^{K=5\times 10^3}$ and a testing dataset $\{ A_{\omega_\ell}\in \mathbb{R}^{33\times 33} \}_{\ell=1}^{L=1\times 10^3}$ are employed without labelled optimal controls in Algorithm~\ref{Alg-RpCtrl-PhysicsInformed}, alongside the same U-Net architecture utilized in section~\ref{Sec-NumExps-RpCtrl-SL-train} (see also \autoref{figure-Unet-SL}). 

Unlike the supervised training and testing losses depicted in \autoref{figure-train-learn-curves-SL}, our unsu- pervised approach proceeds via a label-free optimization. As such, the objective loss values monitored during training are defined as the discretization of functional \eqref{Sto-Ctl-Objective-Strong} over the training $\mathcal{D}_{\text{train}} = \{ A_{\omega_k} \}_{k=1}^K$ and testing $\mathcal{D}_{\text{test}} = \{ A_{\omega_\ell} \}_{\ell=1}^L$ dataset, namely,
\begin{equation*}
\textnormal{Train/Test Loss} = \frac{1}{\vert{}\mathcal{D}\vert{}} \sum_{\omega \in \mathcal{D}} \left( \frac{1}{2} \Vert{} u_{h,\omega}-U \Vert{}_2^2 + \frac{\beta}{2} \Vert{} f_{h,\omega}^{\ms\theta} \Vert{}_2^2 \right).
\end{equation*}
We execute six independent runs and show in \autoref{figure-train-learn-curves-unSL} their learning curves (mean $\pm$ standard deviation). Although our model parameters are updated in a supervised manner \eqref{RpCtl-Loss-Unsupervised-MiniBatch}, both training and testing loss trajectories achieve satisfactory convergence, validating the effectiveness of our physics-informed method in minimizing the objective functional \eqref{Sto-Ctl-Objective-Strong} of robust stochastic control problem \eqref{Sto-Ctl-Weak}.

\begin{figure}[H]
\centering
\includegraphics[width=0.6\textwidth]{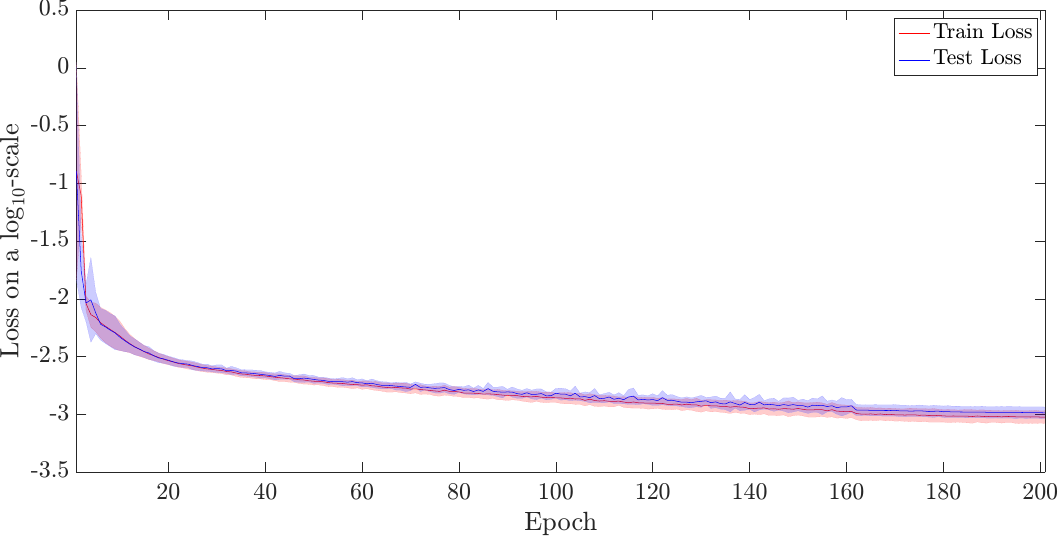}
\vspace*{-0.3cm}
\caption{Unsupervised learning curves (mean $\pm$ standard deviation) over 6 independent runs.}
\label{figure-train-learn-curves-unSL}
\end{figure}

\begin{figure}[H]
\centering
\includegraphics[width=0.6\textwidth]{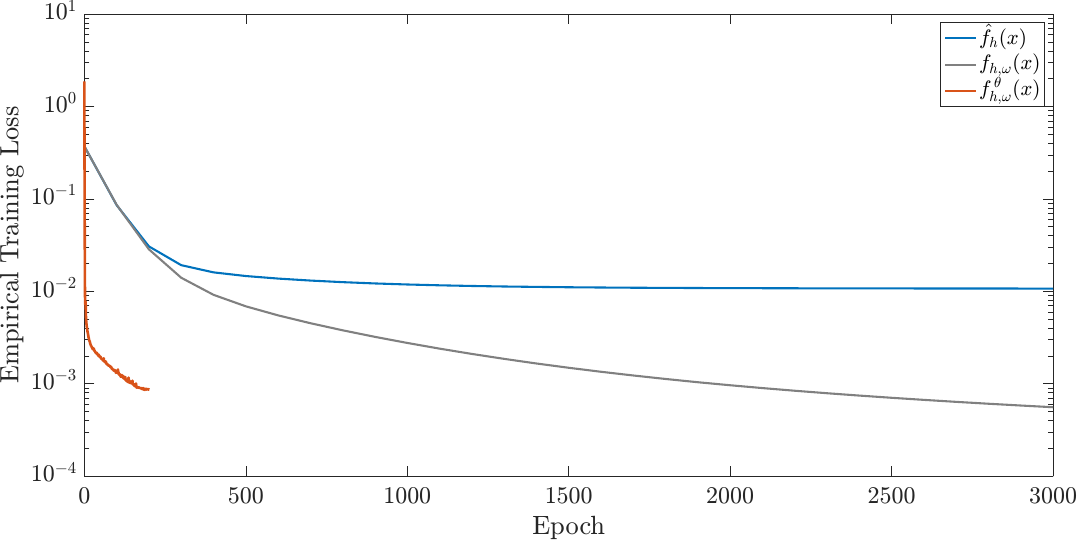}
\vspace*{-0.3cm}
\caption{Comparison of convergence trajectories for empirical objective evaluated against epochs.}
\label{figure-comparison-efficiency-unSL}
\end{figure}

It is also worth noting that our physics-informed surrogate model exploits mini-batch training to achieve rapid generalization across all sample realizations, whereas conventional robust deterministic and pathwise controls require a complete sweep of the full dataset to perform a single control update step. To empirically demonstrate the superior computational efficiency of our learning algorithm, a comparative study against the robust deterministic and pathwise controls is illustrated in \autoref{figure-comparison-efficiency-unSL}, in which a single epoch corresponds to one complete sweep of the training dataset.

Specifically, a single control update in our method requires only a mini-batch of size $B=250$, which roughly corresponds to $B/K=\frac{1}{20}$ steps of classical approaches. As shown in \autoref{figure-comparison-efficiency-unSL}, our proposed method exhibits a drastically steeper reduction in the discrete objective functional, thereby effectively mitigating the computational expense associated with repeated evaluations of the state and adjoint equations.


\subsubsection{Generalization ability} 

To evaluate the generalization capacity of our physics-informed control surrogate, a comparative study against the classical robust deterministic, data-driven robust-pathwise, and pathwise controls is provided in this section. Evaluation is conducted on a testing dataset containing $L = 1 \times 10^3$ unseen coefficient realizations, with numerical results reported in \autoref{figure-test-results-unSL} and \autoref{Table-TestPerformance-RpCtrl-physics-informed}.

As depicted in \autoref{figure-test-results-unSL}, the control predictions generated by our unsupervised surrogate model adapt well to various unseen coefficient realizations, attaining satisfactory target matching performance across the entire testing dataset. Furthermore, \autoref{Table-TestPerformance-RpCtrl-physics-informed} implies that the physics-informed robust-pathwise control achieves accuracy comparable to its data-driven counterpart with considerably reduced offline training time, corroborating our computational efficiency statements discussed in section~\ref{Sec-RpCtrl-physics-informed-algorithm}. On the other hand, our neural operator-based control surrogate is clearly preferable to the robust deterministic control baseline under both supervised and unsupervised learning regimes.

\begin{figure}[t!]
\centering
\includegraphics[width=0.85\textwidth]{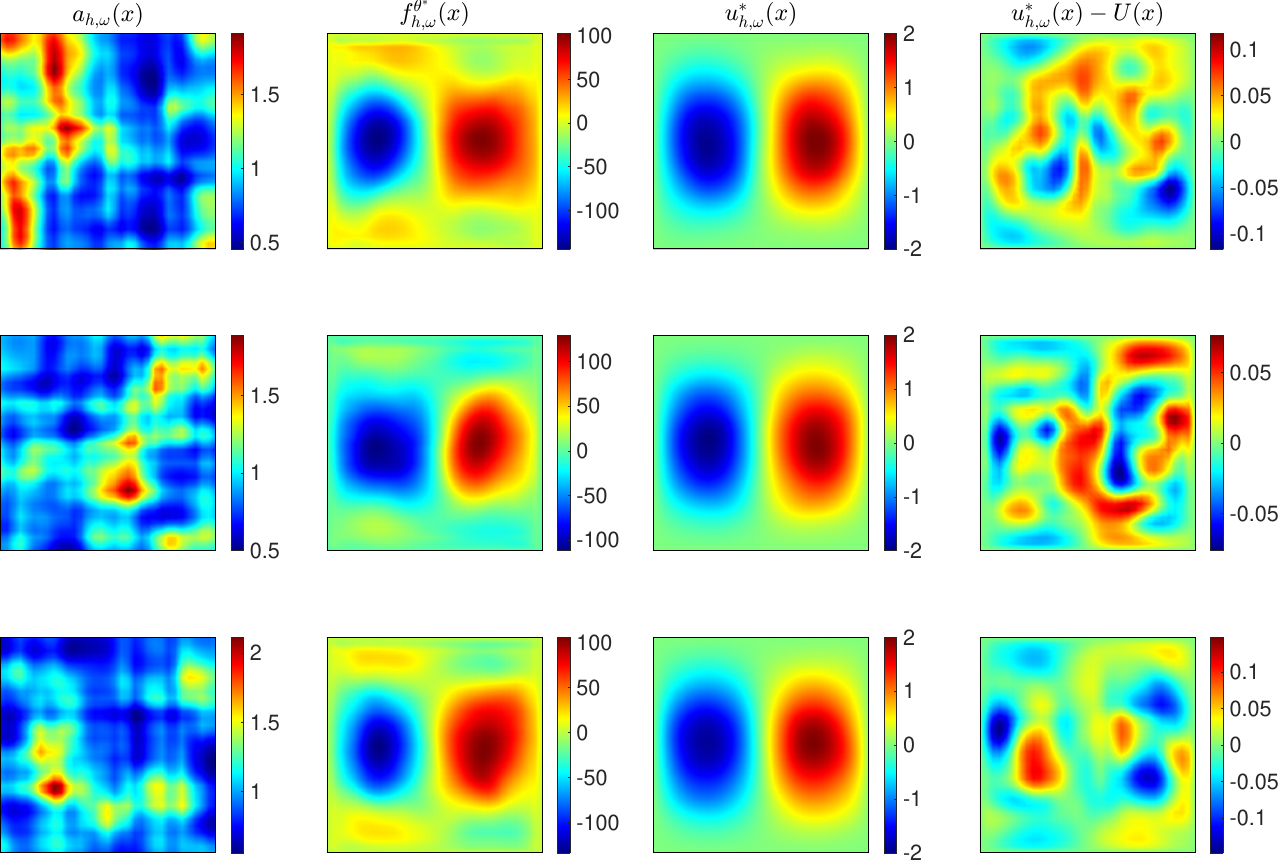}
\vspace*{-0.3cm}
\caption{Left to right: unknown coefficient realizations, physics-driven robust-pathwise controls via trained model, state solutions and their tracking error profiles.}
\label{figure-test-results-unSL}
\end{figure}

\begin{table}[t!]
\centering
\renewcommand{\arraystretch}{1.4}
\caption{Generalization performance of different control approaches  (offline runtime of our data-driven robust-pathwise control consists of data generation 24149.86s and network training 1719.10s). }
\label{Table-TestPerformance-RpCtrl-physics-informed}
\resizebox{\textwidth}{!}{
    \begin{tabular}{ | c | c || c | c | c | c | } 
    \hline
    \multicolumn{2}{|>{\columncolor{kugray5}}c||}{Generalization Test} & \multirow{2}{*}{$\hat{f}^*_h(x)$} &  \multicolumn{2}{c|}{$\mathcal{N}_{\theta^*}(a_{h,\omega}(x))$} & \multirow{2}{*}{$f^*_{h,\omega}(x)$} \\
    \cline{4-5} 
    \multicolumn{2}{|>{\columncolor{kugray5}}c||}{$\ell \!=\! 0.2$, $\sigma \!=\! 10$, $M\!=\!166$} &  & phys-driven & data-driven & \\    
    \hline   					    	
    offline & runtime (s) & 17235.70 & 2791.18 & 25868.96 & 0 \\
    \hline   
    \multirow{3}{*}{online} & runtime (s) & 0  & 0.62 & 0.62 & 4594.66  \\
    \cline{2-6} 
    	& tracking error & $1.35\times10^{-1}$  & $4.09\times10^{-2}$ & $3.29\times10^{-2}$ & $3.26\times10^{-2}$  \\
    	\cline{2-6} 
    	& control effort & 45.29 & 48.91 & 48.52 & 48.48  \\    
    \hline   					    	   						                                                                    
  \end{tabular} 
}
\end{table}

Notably, while the data-driven robust-pathwise control deploys the AdamW optimizer for training, its control targets are determined through a first-order gradient descent scheme, which ultimately limits the accuracy of trained models. In contrast, the optimization of the weakly supervised loss function \eqref{RpCtl-Loss-Unsupervised} easily accommodates acceleration techniques beyond the vanilla mini-batch gradient descent. 


\section{Conclusion}

In this paper, a neural operator-based control strategy is proposed for solving the elliptic optimal control problem with high-dimensional random input data, alongside both data-driven and physics-informed training algorithms for its practical implementation. By directly learning the solution mapping from coefficient realizations to their optimal pathwise control functions, our methods integrate robustness and adaptivity within a single end-to-end surrogate model, obviating the need to compromise between deterministic and stochastic control formulations. Our data-driven training algorithm offers a direct and tractable process, but necessitates extensive high-quality input-output data pairs. To circumvent the data dependency, a physics-informed unsupervised training algorithm is then proposed and recast into a supervised fashion using pseudo labels produced by adjoint-based control updates. Numerical experiments involving hundreds of stochastic dimensions are presented to demonstrate the effectiveness and efficiency of trained surrogate models, showcasing superior tracking accuracy over the robust deterministic control within seconds.

Future work will focus on extending our proposed data-driven/physics-informed robust-pathwise control strategies to broader classes of parametric PDE-constrained optimization problems in scientific and engineering applications, alongside numerical and theoretical investigations into higher-order iterative schemes. In particular, this study adopts a negligible regularization parameter $\beta=10^{-10}$ that leaves the control function nearly unconstrained, motivating future research into the imposition of box constraints to enforce explicit bounds on the control function.


\section*{Acknowledgements}

Qi Sun would like to thank Bin Dong for his support during the early stages of this research. The computations were done on the high performance computers of School of Mathematical Sciences, Tongji University.


\bibliographystyle{siamplain}
\bibliography{references}
\end{document}